\documentclass[a4paper,11pt]{article}

\usepackage[latin1]{inputenc}
\usepackage{amsmath}
\usepackage{amsfonts}
\usepackage{amssymb}
\usepackage{dsfont}

\newtheorem{set2}{Satz}[section]
\newtheorem{theorem}[set2]{Theorem}

\newtheorem{corollary}[set2]{Corollary}

\newtheorem{definition}[set2]{Definition}
\newtheorem{example}[set2]{Example}

\newtheorem{lemma}[set2]{Lemma}
\newtheorem{notation[set2]}{Notation}

\newtheorem{remark}[set2]{Remark}

\newcommand{\ep}{\hfill{$\square$}}
\newenvironment{proof}[1][Proof]{\textbf{#1.} }{\
\\}

\def\XXint#1#2#3{{\setbox0=\hbox{$#1{#2#3}{\int}$}
\vcenter{\hbox{$#2#3$}}\kern-.5\wd0}}

\newcommand{\bl}{\left(}
\newcommand{\br}{\right)}
\newcommand{\e}{\varepsilon}

\newcommand{\di}{\,\mathrm{div}}
\newcommand{\DIV}{\,\mathrm{div}}
\newcommand{\R}{\mathbb{R}}
\newcommand{\N}{\mathbb{N}}
\newcommand{\C}{\mathcal}
\newcommand{\ol}{\overline}
\newcommand{\ul}{\underline}
\newcommand{\dx}{\,\mathrm dx}

\newcommand{\dt}{\,\mathrm dt}
\newcommand{\ds}{\,\mathrm ds}

\newcommand{\dxs}{\,\mathrm dx\,\mathrm ds}

\newcommand{\CC}{\mathbf{C}}
\newcommand{\DD}{\mathbf{D}}

\newcommand{\ukb}{u_{\tau,\beta}^k}

\newcommand{\chikb}{\chi_{\tau,\beta}^k}

\newcommand{\uk}{u^k}
\newcommand{\ukk}{u^{k-1}}
\newcommand{\ukkk}{u^{k-2}}

\newcommand{\chik}{\chi^k}
\newcommand{\chikk}{\chi^{k-1}}

\newcommand{\vk}{v^k}
\newcommand{\vkk}{v^{k-1}}

\newcommand{\Rn}{\mathbb R^n}
\newcommand{\Rnn}{\mathbb R^{n\times n}}

\newcommand{\bu}{\text{\boldmath$u$}}
\newcommand{\bchi}{\text{\boldmath$\chi$}}
\newcommand{\buu}{\text{\boldmath$u$}^0}
\newcommand{\bvv}{\text{\boldmath$v$}^0}
\newcommand{\bchii}{\text{\boldmath$\chi$}^0}
\newcommand{\bb}{\text{\boldmath$b$}}
\newcommand{\bell}{\text{\boldmath$\ell$}}

\allowdisplaybreaks

\date{submitted: February 2, 2015\\updated: September 15, 2016}

\author{M.~Hassan Farshbaf-Shaker$^*$,
Christian Heinemann\footnote{Weierstrass Institute for Applied Analysis and Stochastics (WIAS), Mohrenstr. 39, 10117 Berlin (Germany)
	E-mail: \texttt{hassan.farshbaf-shaker@wias-berlin.de}
	and \texttt{christian.heinemann@wias-berlin.de}
}
\footnote{This work is partially supported by the Research Center Matheon in Berlin (Germany).}}

\begin{document}
\title{
A phase field approach for optimal boundary control of damage processes in two-dimensional viscoelastic media
}

\maketitle

\begin{abstract}
In this work we investigate a phase field model for damage processes in
two-di\-men\-sional viscoelastic media with nonhomogeneous Neumann data
describing external boundary forces.
In the first part we establish global-in-time existence, uniqueness, a priori estimates 
and continuous dependence of strong solutions on the data.
The main difficulty is caused by the irreversibility of the phase field variable, which
results in a constrained PDE system.
In the last part we consider an optimal control problem where a cost
functional penalizes maximal deviations from prescribed damage profiles.
The goal is to minimize the cost functional
with respect to
exterior forces acting on the boundary which play the role of the
control variable in the considered model.
To this end, we prove existence of minimizers and study a family of
``local'' approximations via adapted cost functionals.
\end{abstract}
\\[2mm]
{\it Keywords:} {damage processes, phase field model, viscoelasticity, nonlinear parabolic inclusions,
well-posedness, optimal control.  \\[2mm]
}

\section{Introduction}
	Damage phenomena in elastically deformable solids and their analytical studies have received a lot of attention in the mathematical literature,
	e.g., \cite{BSS05,FM98,FKS99,Gia05,WIAS1520,KRZ11,DT02,Mielke06,MT10,RR14}.
	Especially models which employ a phase field approach and incorporate higher-order terms
	were the focus in some recent works.
	In that case, an internal variable
	indicates the degree of structural integrity and, depending on the material and the scaling,
	may be defined as the volume or the surface density of microvoids or microcracks, respectively, as pointed out in \cite{LD05}.
	This approach has also been utilized for approximations of surface discontinuities
	occurring in the displacement field of fracture models and turned out to be very useful for
	numerical implementations (see \cite{AT90,BOS13,Gia05,PAMM11}).

	One of the main difficulty for a rigorous mathematical investigation of the underlying PDE systems is that the damage variable
	is forced to be monotonically decreasing in time (irreversibility)
	and, if possible, bounded in the unit interval.
	This kind of non-smooth evolution had motivated different concepts of weak solutions and regularization techniques
	in the literature (cf., e.g., \cite{BS04, KRZ11, Mielke06}).
	However, to the authors' best knowledge, a global-in-time well-posedness result for strong solutions with
	inhomogeneous boundary data was left open.
	Together with sufficiently strong a priori estimates
	such a result could be exploited to study optimal control problems typically arising in engineering problems
	focused on resistance against damage and failure. The following model problems with boundary control illustrate some practical examples:
	\begin{itemize}
		\item[--]
			Suppose that a workpiece is exposed to external forces during an experiment and that
			certain parameters related to those forces can be controlled.
			A control problem could be to choose optimal parameters in order to prevent further
			damage in the material.
		\item[--]
			Related to the first scenario we might
			be interested in calculating additional forces not to prevent but to redirect crack spreading to non-critical components of the structure
			and to avoid complete failure.
		\item[--]
			Another problem might be the determination of external forces in order to
			deliberately induce a damage progression.
			For instance, it might be desirable to separate certain parts of the workpiece in industrial processes.
	\end{itemize}
	
	By now, to the authors' best knowledge, the mathematical contributions addressing those and related problems are inspired by the pioneering work \cite{De89} and employ fracture models
	to control the energy release rate of a single crack in a quasi-stationary setting by optimal shape design techniques,
	fibers or applied forces (see \cite{HMO08, KLS12,LSZ15,PLSS11} for more details).
	The cracks are explicitely modeled by non-smooth domains with or without non-penetration conditions for the deformation.
	A main issue consists in determining optimal forces or inclusions in the solid in order to cease crack propagation
	or to release as much energy as possible.

	In this paper we would like to advance a different approach for such control problems by utilizing a phase field model for damage.
	The kind of model under consideration was motivated by \textsc{Fr\' emond} and \textsc{Nedjar} in \cite{FN96} and is stated below.
	Under certain structural assumptions we are able to investigate well-posedness of strong solutions and existence of optimal boundary controls
	for a coupled evolutionary system describing damage processes in viscoelastic materials in two spatial dimensions.
	This enables us to investigate optimal control problems where
	we aim to control the damage phase field variable via external boundary forces.
	The cost functional will measure maximal deviations from desired phase field profiles.
	
	
	
	In the first part of this paper we study existence and then, for constant viscosity $\mathbb D$, well-posedness of the following PDE problem:\\
	\textit{For a given time interval $(0,T)$ and reference configuration $\Omega$ with boundary $\Gamma$
	and outer unit normal $\nu$, find $(u,\chi)$ such that}
	\begin{subequations}
	\label{eqn:pdeSystem}
	\begin{align}
	\label{eqn:elasticEq}
		&u_{tt}-\di\big(\mathbb C(\chi)\e(u)+\mathbb D(\chi)\e(u_t)\big)=\ell
		&&\textit{ in }\Omega\times(0,T),\\
	\label{eqn:damageEq}
		&\chi_t-\Delta\chi_t-\Delta\chi+\xi+\frac 12 \mathbb C'(\chi)\e(u):\e(u)+f'(\chi)=0
		&&\textit{ in }\Omega\times(0,T)
	\end{align}
	\end{subequations}
	\textit{with the subgradient}
	\begin{align}
	\label{eqn:damageEq2}
		&\xi\in\partial I_{(-\infty,0]}(\chi_t)\hspace*{16em}
		&&\textit{ in }\Omega\times(0,T)
	\end{align}
	\textit{and the initial-boundary conditions}
	\begin{subequations}
	\label{eqn:IBC}
	\begin{align}
	\label{eqn:initialCond}
		&u(0)=u^0,\;u_t(0)=v^0,\;\chi(0)=\chi^0
		&&\textit{ in }\Omega,\\
	\label{eqn:boundaryEq}
		\hspace*{3.0em}&\big(\mathbb C(\chi)\e(u)+\mathbb D(\chi)\e(u_t)\big)\cdot\nu =b\hspace*{13.0em}
		&&\textit{ on }\Gamma\times(0,T),\hspace*{1.0em}\\
	\label{eqn:boundaryEq2}
		&\nabla(\chi+\chi_t)\cdot\nu =0
		&&\textit{ on }\Gamma\times(0,T).
	\end{align}
	\end{subequations}
	Equation \eqref{eqn:elasticEq} describes the balance of forces in the workpiece according to the Kelvin-Voigt rheology.
	The displacement field is denoted by $u$, the external volume forces by $\ell$, the linearized strain tensor by $\e(u)=\frac12(\nabla u+(\nabla u)^T)$ and
	the stress tensor by $\sigma=\mathbb C(\chi)\e(u)+\mathbb D(\chi)\e(u_t)$.
	The first summand of $\sigma$ contains the elastic contribution whereas the second summand models viscous effects.
	The coefficient $\mathbb C$ designates the fourth-order damage-dependent stiffness tensor and $\mathbb D$ the viscosity tensor.
	The second equation \eqref{eqn:damageEq} specifies the parabolic evolution law for the propagation of damage described
	by the variable $\chi$ under the constraint \eqref{eqn:damageEq2},
	where the subdifferential of the indicator function $I_{(-\infty,0]}:\R\to\R\cup\{\infty\}$ is given by
	\begin{align*}
		\partial I_{(-\infty,0]}(\chi_t)=
		\begin{cases}
			\{0\}&\text{if }\chi_t<0,\\
			[0,\infty)&\text{if }\chi_t=0,\\
			\emptyset&\text{if }\chi_t>0.
		\end{cases}
	\end{align*}
	The Laplacians $-\Delta\chi$ and $-\Delta\chi_t$ model diffusive effects of $\chi$ and $\chi_t$
	and have a regularizing effect from the mathematical perspective.
	For a mechanical motivation of system \eqref{eqn:pdeSystem}-\eqref{eqn:IBC} by means of balance laws and constitutive relations we refer to \cite{Fre02, Fr12, FN96}.
	In comparison to certain phase field models of damage used in the literature, the
	higher-order viscosity $-\Delta\chi_t$ is also incorporated in the damage law \eqref{eqn:damageEq} (cf. \cite{BS04, bobo2}).
	It will help us to perform the (global-in-time) \textit{Second a priori estimate} in Lemma \ref{lemma:aPrioriDiscr}
	(see also the remark after the proof of Theorem \ref{theorem:existenceCont})
	and, from the modeling point of view, it originates from an additional gradient term
	of $\chi_t$ in the dissipation potential
	corresponding to system \eqref{eqn:pdeSystem}, which is given by
	\begin{align}
		R(u_t,\chi_t)=\int_\Omega\Big(\frac{1}{2}|\chi_t|^2+\frac{1}{2}|\nabla\chi_t|^2+\frac{1}{2}\mathbb D(\chi)\e(u_t):\e(u_t)+I_{(-\infty,0]}(\chi_t)\Big)\dx.
	\end{align}
	We would like to give the following interpretation for the subgradient constraint \eqref{eqn:damageEq2}:
	
	By introducing the free energy $F$ to system \eqref{eqn:pdeSystem} as
	\begin{align}
		F(u,\chi)=\int_\Omega\Big(\frac 12|\nabla\chi|^2+\frac 12\mathbb C(\chi)\e(u):\e(u)+f(\chi)\Big)\dx,
	\label{eqn:freeEnergy}
	\end{align}
	we may rewrite \eqref{eqn:damageEq} as\vspace*{-0.3em}
	$$
		0\in\partial_{\chi_t}R(u_t,\chi_t)+d_\chi F(u,\chi)\quad\text{or, equivalently,}\quad
		\xi=-\chi_t+\Delta\chi_t-d_\chi F(u,\chi).\vspace*{-0.25em}
	$$
	with a complementarity formulation for \eqref{eqn:damageEq2}, i.e.
	$$
		\chi_t\leq 0,\quad \xi\cdot\chi_t=0,\quad \xi\geq 0.
	$$
	
	We comment that for a mechanical interpretation it would be desirable
	to force the values of the phase field variable $\chi$ to remain in the unit interval $[0,1]$.
	Principally, that can be achieved by incorporating an indicator
	function in the energy \eqref{eqn:freeEnergy} (or another suitable potential) or by using, if available, a parabolic maximum principle
	in the sense that $\chi(0)\in [0,1]$ a.e. in $\Omega$ implies $\chi(t)\in [0,1]$ a.e. in $\Omega$
	for almost all times $t\in(0,T)$.
	However, the first option would lead to the occurrence of a second subgradient $\varphi\in\partial I_{[0,1]}(\chi)$ in the
	damage equation \eqref{eqn:damageEq} which would
	possibly violate uniqueness of solutions (a behavior observed in \cite{CV90}
	for certain double inclusions; see also \cite[Remark 2.18]{RR12})
	whereas the second option leads to serious complicacies in establishing the maximum principle
	due to the regularizing term $-\Delta\chi_t$ in \eqref{eqn:damageEq}.
	If the tensors $\mathbb C(\cdot)$ and $\mathbb D(\cdot)$ are constant
	for non-positive values we can at least provide a pragmatical solution
	(see also Example \ref{example:cFunction}):
	
	To this end, let $(u,\chi)$ be a solution to the above problem
	and let the initial data satisfy $\chi^0\in[0,1]$ in $\Omega$.
	Then together with the irreversibility constraints $\chi_t\in(-\infty,0]$ we find $\chi\in (-\infty,1]$
	a.e. in $\Omega\times(0,T)$.
	Let us consider the pair $(u,\chi^+)$ with the pointwise truncation $\chi^+:=\max\{\chi,0\}$ which
	thus satisfies $\chi^+\in[0,1]$ a.e. in $\Omega\times(0,T)$.
	Now note that $(u,\chi^+)$ still fulfills the elasticity equation \eqref{eqn:elasticEq}
	in $\Omega\times(0,T)$ when replacing $\chi$ by $\chi^+$
	because $\mathbb C(\chi)=\mathbb C(\chi^+)$ and $\mathbb D(\chi)=\mathbb D(\chi^+)$ for all values $\chi$.
	The damage equation \eqref{eqn:boundaryEq} when replacing $\chi$ by $\chi^+$
	is satisfied in the subset $\{\chi >0\}\;\subseteq\;\Omega\times(0,T)$.
	In the complementary part $\{\chi \leq 0\}\;\subseteq\;\Omega\times(0,T)$ 
	the damage evolution for $\chi^+$ ceases and stays at the minimum, i.e. $\chi^+=0$ and thus (in an a.e. sense) $(\chi^+)_t=0$.

	The second part of this paper is devoted to an optimal control problem.
	A cost functional $\C J$ will measure the maximal deviation of the damage variable $\chi$ from given prescribed damage profiles
	at the final time $T$ and/or at all times in $[0,T]$ ($\lambda_Q,\lambda_\Omega,\lambda_\Sigma\geq 0$):
	\begin{align}
		\C J(\chi,b):={}&\frac{\lambda_Q}{2}\|\chi-\chi_Q\|_{L^\infty(\Omega\times(0,T))}
			+\frac{\lambda_\Omega}{2}\|\chi(T)-\chi_T\|_{L^\infty(\Omega\times(0,T))}
			+\frac{\lambda_\Sigma}{2}\|b\|_{L^2(\Gamma\times(0,T);\Rn)}^2.
	\label{eqn:introJ}
	\end{align}
	A minimizer $(\chi,b)$ of $\C J$
	under the constraint that $\chi$ solves
	system \eqref{eqn:pdeSystem}-\eqref{eqn:IBC} for some displacement $u$ and admissible
	(later specified) boundary data $b$
	indicates an evolution which approximates $\chi_Q$ and/or $\chi_T$
	best in the sense of $\C J$.
	It is also possible to replace $\chi$ by $\chi^+$ on the right-hand side of \eqref{eqn:introJ}
	in order to account only for the non-negative values of $\chi$. 
		
	
	In the following we summarize the main results of our paper:
	\begin{itemize}
		\item[--]
			In Theorem \ref{theorem:existenceCont} we will prove existence of strong solutions
			for system \eqref{eqn:pdeSystem}-\eqref{eqn:IBC} and for a so-called
			$\beta$-approximation in two spatial dimensions.
			In the latter case we replace the subgradient $\xi$ in \eqref{eqn:damageEq2} by a
			smooth approximation $\xi_\beta(\chi_t)$ with $\beta>0$.
			On the one hand this enables us to perform the a priori estimates in Lemma
			\ref{lemma:aPrioriDiscr}, while, on the other hand,
			the $\beta$-approximation might be helpful for further studies such as optimality
			systems, numerical implementations etc.
			We emphasize that the existence analysis constitutes the main part of this paper and
			strongly relies on the two-dimensional
			\textit{Ladyzhenskaya's inequality} originally devised for the
			$2D$ Navier-Stokes equations
			(see \cite{Lad58} and the calculation \eqref{eqn:GN}).
		\item[--]
			Continuous dependence on the data $(u^0,v^0,\chi^0,b,\ell)$ and, in particular,
			uniqueness of strong solutions for system \eqref{eqn:pdeSystem}-\eqref{eqn:IBC}
			are proven in Theorem \ref{theorem:contDepCont} (see also Corollary \ref{cor:uniqueness})
			under the assumption of constant viscosity $\mathbb D$.
			We also establish a priori estimates for the solutions in Corollary \ref{cor:aPriori}.
			These results allow us to define the solution operator and constitutes the fundament for
			the considered optimal control problem.
		\item[--]
			Theorem \ref{theorem:optimalControl} reveals existence to an optimal control problem
			where the cost functional penalizes deviations
			of the damage variable from given damage profiles in the $L^\infty$-norm (see
			\eqref{eqn:introJ}).
			The strong solutions of system \eqref{eqn:pdeSystem}-\eqref{eqn:IBC} will be controlled via
			external boundary forces.
			We prove existence of optimal controls by using the $\beta$-approximation in the proof of
			Theorem \ref{theorem:existenceCont} to define a family of optimal control problems.
			The minimizers or the optimal controls of the family of $\beta$-approximating control
			problems converge in a limit process (along a subsequence as $\beta\downarrow 0$) to an
			optimal control of the original control problem. In other words, we show that optimal
			controls for the family of $\beta$-approximating control problems are for some $\beta>0$
			likely to be ``close'' to optimal controls for the original control problem. It is natural
			to ask if the reverse holds, i.e., whether every optimal control for the original control
			problem can be approximated by a sequence of optimal controls of the $\beta$-approximating
			control problems. 
				Unfortunately, we will not be able to prove such a ``global result'' that applies to all
				optimal controls for the original control problem. The reason for that lies on the
				non-convexity of the optimal control problems (both the original one and the
				$\beta$-approximating control problems) and consequently on the non-uniqueness of the
				optimal controls.
				However, a ``local'' result can be established by introducing so-called adapted optimal
				control problems in Theorem \ref{theorem:adaptedOCP}.
	\end{itemize}
	Let us recall some already established results in the mathematical literature of
	phase field models for damage/gradient-of-damage models:
	\begin{itemize}
		\item[--]
			Local-in-time well-posedness of strong solutions for damage-elasticity systems with
			scalar-valued displacements, homogeneous Dirichlet conditions for the displacements
			and linear dependence of $\mathbb C$ and $\mathbb D$ on $\chi$
			is proven in \cite{BS04,BSS05} and in \cite{FKNS98,FKS99} for one-dimensional models.
			Since a degenerating elastic energy with respect to $\chi$ is considered and
			enhanced estimates are established for small times,
			the solutions are obtained locally in time.
			In the work \cite{bobo2} such rate-dependent damage-elasticity system
			with the additional second order term $\chi_{tt}$ has been
			coupled with an equation for heat conduction incorporating highly non-linear terms.
			Existence, uniqueness and
			regularity results are established locally in time.

			
		\item[--]
			Rate-independent gradient-of-damage models are explored in \cite{Mielke06} and in
			subsequent papers, e.g., \cite{Mie11, Mielke10}.
			In that case the corresponding dissipation potential $R$ is positive, convex, and positively-homogeneous
			of degree 1.
			From the modeling point of view it means that the damage progression is considered on a
			faster timescale than the acting of the external forces.
			The authors considered non-smooth domains and employed weak notions referred to as \textit{energetic formulation} in order to prove existence of solutions.
			The degenerating case where the material may loose all its elastic properties due to heavy damage is also studied.
			Further cases involving nonlinear $r$-Laplacians with $r>1$ or even $r=1$ instead of the classical Laplacian
			in the damage equation are investigated in \cite{MT10,Th13},
			where also higher temporal regularity is shown.
			
		\item[--]
			A weak notion for rate-dependent damage models coupled with Cahn-Hilliard equations was
			introduced in \cite{WIAS1520} for quasi-static balance of forces and in \cite{WIAS1919} with inertial effects and without the viscosity term in \eqref{eqn:elasticEq}.
			Existence of weak solutions is proven there for non-smooth domains and mixed-boundary conditions for the displacements
			whereas uniqueness is left open.
			
		\item[--]
			Well-posedness and vanishing viscosity results for damage models with (nonlocal) higher-order $s$-Laplacian 
			are established in \cite{KRZ11} (see also \cite{KRZ13} for vanishing viscosity
			results for damage models with a regularized nonlinear $q$-Laplacian in non-smooth settings).
			The authors study both viscous (rate-dependent)
			and, via the vanishing viscosity limit, rate-independent PDE systems for damage.
			Existence of solutions is provided in both situations
			where in the latter case
			a novel energetic formulation making use of arc-length reparameterization
			in order to describe the behavior of the system at jumps is utilized.
			In particular, the existence results cover the case of the standard Laplacian
			in two spatial dimensions.
			Furthermore, the uniqueness problem has been solved under special conditions, where
			in the case of two or three spatial dimensions
			the $s$-Laplacian is assumed to be of higher-order than the classical Laplacian.
			
		\item[--]
			Coupled thermoviscoelastic and isothermal damage models incorporating $p$-Laplacian
			operators are analyzed in \cite{RR12} (see also
			\cite{RR14} for the full heat equation including all dissipative terms and
			\cite{WIAS1927} for damage-dependent heat expansion coefficients).
			In those works homogeneous Dirichlet boundary conditions for the displacements are
			assumed.
			Uniqueness is shown in the isothermal case by adopting $p>n$ and by dropping the
			irreversibility constraint \eqref{eqn:damageEq2}.
			Existence results for the corresponding rate-independent thermoviscoelastic damage models are
			proven in the recent paper \cite{LRTT14}.
	\end{itemize}
	\textbf{Structure of the paper}\\
	Section \ref{section:wellposedness} is devoted to the well-posedness problem of system \eqref{eqn:pdeSystem}-\eqref{eqn:IBC}.
	We state the precise assumptions in Subsection \ref{sec:assumptions} and introduce time-discretized and $\beta$-regularized approximations
	of \eqref{eqn:pdeSystem}-\eqref{eqn:IBC} in Subsection \ref{section:notion}.
	The existence proofs are carried out in Subsection \ref{sec:existence} firstly for the time-discretized and then, by a limit analysis,
	for the time-continuous versions.
	In the final part of that section, i.e. in Subsection \ref{sec:contDep}, we prove continuous dependence on the initial-boundary data.
	Then, equipped with the well-posedness result, we state the announced optimal control problem in Section \ref{section:OCP}.
	We prove existence of optimal controls via $\beta$-regularization in Subsection \ref{section:existenceOCP}
	and their approximation by means of an adapted cost functional in Subsection \ref{section:adaptedOCP}.

\section{Analysis of the evolution inclusions}
\label{section:wellposedness}
	The approach presented in this work
	combines two different approximation techniques	to obtain
	existence of solutions for system \eqref{eqn:pdeSystem}-\eqref{eqn:IBC}: semi-implicit time-discretization and
	regularization of the subgradient $\xi$ in \eqref{eqn:damageEq2}.
	At first we will tackle the existence problem for the time-discrete and regularized system in Lemma \ref{lemma:existenceDiscr}.
	By passing the discretization fineness to $0$, solutions of a time-continuous regularized system are obtained in Theorem \ref{theorem:existenceCont} (i).
	In the final step, a further limit passage leads to solutions of the desired limit system (see Theorem \ref{theorem:existenceCont} (ii)).
	Then, we conclude this section in Theorem \ref{theorem:contDepCont} with a uniqueness and continuous dependence result.

\subsection{Assumptions and notation}
\label{sec:assumptions}
	Throughout this work, we adopt the following assumptions:
	\begin{enumerate}
		\item[\textbf{(A1)}]
			$\Omega\subseteq\R^n$ with $n\in\{1,2\}$ is a bounded $C^2$-domain. The boundary is denoted by
			$\Gamma$ and the outer unit normal by $\nu$.
		\item[\textbf{(A2)}]
			The damage-dependent stiffness tensor satisfies $\mathbb C(\cdot)=\mathsf{c}(\cdot)\CC$, where
			the coefficient function $\mathsf{c}$ is assumed to be
			of the form
			\begin{align*}
				\mathsf{c}=\mathsf{c}_1+\mathsf{c}_2\text{ where $\mathsf{c}_1\in C^{1,1}(\R)$ is convex and $\mathsf{c}_2\in C^{1,1}(\R)$ is concave.}
			\end{align*}
			Moreover, we assume that $\mathsf{c},\mathsf{c}_1',\mathsf{c}_2'$ are bounded and
			as well as
			\begin{align*}
				\mathsf{c}(x)\geq 0\qquad\text{ for all }x\in\R.
			\end{align*}
			
			The 4$^\mathrm{th}$ order stiffness tensor $\CC\in\C L(\R_\mathrm{sym}^{n\times n};\R_\mathrm{sym}^{n\times n})$
			is assumed to be symmetric and positive definite, i.e.
			\begin{align}
				\CC_{ijlk}=\CC_{jilk}=\CC_{lkij}\text{ and }e:\CC e\geq \eta|e|^2\text{ for all }e\in \R_\mathrm{sym}^{n\times n}
					\label{eqn:Ctensor}
			\end{align}
			with constant $\eta>0$.
		\item[\textbf{(A3)}]
			The damage-dependent viscosity tensor satisfies $\mathbb D(\cdot)=\mathsf{d}(\cdot)\DD$, where
			the coefficient function $\mathsf{d}$ satisfies $\mathsf{d}\in C^1(\R)$.
			Moreover, we assume that $\mathsf{d}$ and $\mathsf{d}'$ are bounded and
			\begin{align}
			\label{eqn:incompleteDamage}
				\text{$\mathsf{d}(x)\geq\eta>0$ for all $x\in\R$ and fixed $\eta>0$.}
			\end{align}
			The 4$^\mathrm{th}$ order tensor $\DD$ is given by $\DD=\mu\CC$, where $\mu>0$ is a constant.
		\item[\textbf{(A4)}]
			The damage-dependent potential function $f$ is assumed to be in $f\in C^{1,1}(\R)$.
	\end{enumerate}
	In the assumptions above $C^{1,1}(\R)$ should be understood as
	the space of differentiable functions on $\R$
	whose derivatives are Lipschitz continuous.
	\begin{example}
	\label{example:cFunction}
		Let $\widetilde{\mathsf{c}}\in C^{1,1}([0,1])$ be any given non-negative function
		with $(\widetilde{\mathsf{c}})'(0)=0$.
		Then there exists an extension of $\widetilde{\mathsf{c}}$ to the entire real line $\R$
		(the extension is denoted by $\mathsf{c}$) such that $\mathsf{c}(x)=\mathsf{c}(0)$ for all $x<0$ and
		assumption (A2) is satisfied for $\mathsf c$ and a given tensor $\CC$ with \eqref{eqn:Ctensor}.
	\end{example}
	\begin{proof}[Proof of Example \ref{example:cFunction}]
		We define the convex function $\widetilde{\mathsf{c}}_1$ and the concave
		function $\widetilde{\mathsf{c}}_2$ such that $\widetilde{\mathsf{c}}=\widetilde{\mathsf{c}}_1+\widetilde{\mathsf{c}}_2$
		on the compact intervall $[0,1]$ via
		\begin{align*}
			&\widetilde{\mathsf{c}}_1(x):= \widetilde{\mathsf{c}}(0)+\int_0^x\Big(\int_0^s \max\{(\widetilde{\mathsf{c}})''(\tau),0\}\,\mathrm d\tau\Big)\ds,\\
			&\widetilde{\mathsf{c}}_2(x):= \int_0^x\Big(\int_0^s \min\{(\widetilde{\mathsf{c}})''(\tau),0\}\,\mathrm d\tau\Big)\ds.
		\end{align*}
		
		To extent the functions $\widetilde{\mathsf{c}}_1$ and $\widetilde{\mathsf{c}}_2$ from the domain $[0,1]$
		to $\R$ in accordance with (A2) we have to be careful because
		the extensions should have bounded derivatives
		as well as bounded second derivatives whereas the sum of the extensions
		should also be bounded.
		We provide the following construction:
		
		Since $(\widetilde{\mathsf{c}}_1)'\geq 0$, $(\widetilde{\mathsf{c}}_1)''\geq 0$
		as well as $(\widetilde{\mathsf{c}}_2)'\leq 0$ and $(\widetilde{\mathsf{c}}_2)''\leq 0$
		on $[0,1]$ it hold
		\begin{align}
			(\widetilde{\mathsf{c}}_1)'(1)=\max_{y\in[0,1]}|(\widetilde{\mathsf{c}}_1)'(y)|=:\lambda_1,\qquad\qquad
			-(\widetilde{\mathsf{c}}_2)'(1)=\max_{y\in[0,1]}|(\widetilde{\mathsf{c}}_2)'(y)|=:\lambda_2.
				\label{eqn:convConcIdentities}
		\end{align}
		In the case $\lambda_1\leq \lambda_2$ we may extend $\widetilde{\mathsf{c}}_1$ and $\widetilde{\mathsf{c}}_2$
		to $\R$ as follows (for readers' convenience we keep the integrals explicitly):
		\begin{align*}
			&\mathsf{c}_1(x):=
			\begin{cases}
				\widetilde{\mathsf{c}}_1(0)&\text{if }x<0,\\
				\widetilde{\mathsf{c}}_1(x)&\text{if }0\leq x<1,\\
				\widetilde{\mathsf{c}}_1(1)+\int_1^x\lambda_1+(\lambda_2-\lambda_1)\frac{s-1}{\delta}\ds&\text{if }1\leq x<1+\delta,\\
				\widetilde{\mathsf{c}}_1(1)+\int_1^{1+\delta}\lambda_1+(\lambda_2-\lambda_1)\frac{s-1}{\delta}\ds+\lambda_2(x-(1+\delta))&\text{if }1+\delta\leq x,
			\end{cases}\\
			&\mathsf{c}_2(x):=
			\begin{cases}
				\widetilde{\mathsf{c}}_2(0)&\text{if }x<0,\\
				\widetilde{\mathsf{c}}_2(x)&\text{if }0\leq x<1,\\
				\widetilde{\mathsf{c}}_2(1)-\lambda_2(x-1)&\text{if }1\leq x.
			\end{cases}
		\end{align*}
		Here, $\delta>0$ can be chosen arbitrarily.
		Due to \eqref{eqn:convConcIdentities} we observe that $\mathsf{c}_1\in C^{1,1}(\R)$
		and that $\mathsf{c}_1$ is convex in the entire real line $\R$.
		Furthermore, $\mathsf{c}_2\in C^{1,1}(\R)$ is concave in $\R$.
		The particularity of this construction is that $\mathsf{c}_1$ and $\mathsf{c}_2$
		have only linear growth for large $x>0$ with derivatives $\lambda_2$ and $-\lambda_2$,
		respectively. Thus the sum $\mathsf{c}_1+\mathsf{c}_2$ is bounded.
		
		Hence we observe that $\mathsf{c}_1',\mathsf{c}_2'$ and
		$\mathsf{c}_1+\mathsf{c}_2$ are bounded
		and that the properties in (A2) are fulfilled
		for $\mathsf{c}$,	$\mathsf{c}_1$ and $\mathsf{c}_2$.		
		The case $\lambda_1>\lambda_2$ can be treated analogously.\ep
	\end{proof}
	
	\begin{remark}
	\label{remark:assumptions}
		\begin{itemize}
			\item[(i)]
				The non-degeneracy condition \eqref{eqn:incompleteDamage} prevents the material from
				complete damage, i.e., even the maximal damaged parts (the region with $\chi\leq 0$)
				exhibit small viscous properties.
			\item[(ii)]
				The assumption $\DD=\mu\CC$ in (A3) is needed
				in the proof of Lemma \ref{lemma:existenceDiscr} in step 2
				in order to perform 	a regularity argument based on a transformation.
				It has already been employed in the mathematical literature (see \cite{WIAS1927,RR14}).
		\end{itemize}
	\end{remark}

	For later use, we define the solution space $\C U\times\C X$, where $\C U$ denotes the space
	of the displacements and $\C X$ the space of the damage evolutions given by
	\begin{subequations}
	\label{eqn:spaces}
	\begin{align}
		\C U&:=H^1(0,T;H^2(\Omega;\R^n))\cap W^{1,\infty}(0,T;H^1(\Omega;\R^n))\cap H^2(0,T;L^2(\Omega;\R^n)),\\
		\C X&:=H^1(0,T;H^2(\Omega)).
	\end{align}
	\end{subequations}
	The space of boundary controls $\C B$ is defined as
	\begin{align*}
		\C B:=L^2(0,T;H^{1/2}(\Gamma;\Rn))\cap H^1(0,T;L^{2}(\Gamma;\Rn)).
	\end{align*}
	We also introduce the sets for brevity
	\begin{align*}
		Q:=\Omega\times(0,T),\qquad\qquad\Sigma:=\Gamma\times(0,T).
	\end{align*}
	Finally, let us mention that we make frequently use of the standard Young's inequality
	$$
		ab\leq \delta a^2+\frac{1}{4\delta}b^2\quad\text{ for all }a,b\in\R\text{ and all }
		\delta>0
	$$
	where $\delta>0$ will be chosen when necessary and we write $C_\delta:=\frac{1}{4\delta}$.
	Moreover, the symbols $C$, $\widetilde C$, $D$, $\widetilde \eta$ and $\delta$
	will denote positive constants throughout this work.

\subsection{Notions of solution}
\label{section:notion}
	Let us consider two approximations of system \eqref{eqn:pdeSystem}-\eqref{eqn:IBC}:
	a regularized version where the indicator function $I_{(-\infty,0]}$ in \eqref{eqn:damageEq2} is replaced
	by a suitable smooth function $I_\beta$, $\beta\in(0,1)$,
	and a time-discretized version of the regularized system.
	To this end, we introduce the following regularization:
	\begin{definition}[$\beta$-regularization]
	\label{def:regularization}
		Let the family of functions $\{I_\beta\}_{\beta\in(0,1)}\subseteq C^{1,1}(\R)$ denote a regularization of the indicator function $I_{(-\infty,0]}$
		in the following sense:
		\begin{itemize}
		\item[(i)]
			$I_{\beta_1}\leq I_{\beta_2}$ pointwise in $\R$ for every $\beta_1,\beta_2\in(0,1)$ with $\beta_1\geq \beta_2$,
		\item[(ii)]
			$I_\beta\uparrow\infty$ pointwise in $[0,\infty)$ as $\beta\downarrow 0$,
		\item[(iii)]
			$I_\beta(x)= 0$ for all $x\leq 0$ and all $\beta\in(0,1)$,
		\item[(iv)]
			$I_\beta''(x)\geq 0$ for a.e. $x\in\R$ and all $\beta\in(0,1)$.
		\end{itemize}
		We may also write $\xi_\beta:=I_\beta'$ in the following.
	\end{definition}
	\begin{remark}
		In particular, we may choose the Moreau-Yosida approximation given by (see \cite[Lemma 5.17]{Rou13})
		$$
			I_\beta(x)=\inf_{y\in\R}\bigg(\frac{|x-y|^2}{2\beta}+I_{(-\infty,0]}(y)\bigg)
			=
			\begin{cases}
				0&\text{if }x\leq 0,\\
				\frac{1}{2\beta}x^2&\text{if }x> 0.
			\end{cases}
		$$
		Let us mention that also $C^\infty$-approximations may be chosen for $\{I_\beta\}$
		especially in view of optimality systems for optimal control problems (see \cite[Chapter 5]{NT94}).
	\end{remark}
	\begin{definition}[Strong solutions]
	\label{def:notionSolution}
		For system \eqref{eqn:pdeSystem}-\eqref{eqn:IBC} and their approximations
		we introduce the following notion of solutions:
		\begin{enumerate}
			\item[(i)]\textbf{Time-continuous limit system ($\tau=0,\beta=0$).}\\
				Let the data $(u^0,v^0,\chi^0,b,\ell)$ be given.
				A solution of the time-continuous limit system is a pair of functions $(u,\chi)\in\C U\times \C X$
				safisfying \eqref{eqn:pdeSystem}-\eqref{eqn:IBC} in an a.e. sense
				and for a subgradient $\xi\in L^2(Q)$.
			\item[(ii)]\textbf{Time-continuous $\beta$-regularized system ($\tau=0,\beta>0$).}\\
				Let the data $(u^0,v^0,\chi^0,b,\ell)$ be given.
				A solution of the time-continuous $\beta$-regularized system is a pair of functions
				$(u,\chi)\in\C U\times \C X$
				with $u(0)=u^0$, $\partial_t u(0)=v^0$ and $\chi(0)=\chi^0$
				such that
				\begin{subequations}
				\begin{align}
				\label{eqn:elasticRegEq}
					&u_{tt}-\di\big(\mathbb C(\chi)\e(u)+\mathbb D(\chi)\e(u_t)\big)=\ell
					&&\text{a.e. in }Q,\\
				\label{eqn:damageRegEq}
					&\chi_t-\Delta\chi_t-\Delta\chi+\xi_\beta(\chi_t)+\frac 12 \mathbb C'(\chi)\e(u):\e(u)+f'(\chi)=0\hspace*{-0.5em}
					&&\text{a.e. in }Q,\\
				\label{eqn:boundaryRegEq}
					&\big(\mathbb C(\chi)\e(u)+\mathbb D(\chi)\e(u_t)\big)\cdot \nu  = b
					&&\text{a.e. on }\Sigma,\\
				\label{eqn:boundaryRegEq2}
					&\nabla(\chi+\chi_t)\cdot \nu =0
					&&\text{a.e. on }\Sigma.
				\end{align}
				\end{subequations}
			\item[(iii)]\textbf{Time-discrete $\beta$-regularized system ($\tau,\beta>0$).}\\
				Let $\{0,\tau,2\tau,\ldots,T\}$ denote an equidistant partition of $[0,T]$
				with discretization fineness $\tau:=T/M$ and $M\in\N$.
				Furthermore, let the data $(u^{0},u^{-1},\chi^0)$,
				$\{b^k\}_{k=0,\ldots,M}$ as well as $\{\ell^k\}_{k=0,\ldots,M}$ be given.
				A solution of the time-discrete $\beta$-regularized system is a sequence
				$\{u^k,\chi^k\}_{k=0,\ldots,M}$ of functions
				$u^k\in H^2(\Omega;\R^n)$ and $\chi^k\in H^2(\Omega)$ such that
				\begin{subequations}
				\begin{align}
					&\frac{u^k-2u^{k-1}+u^{k-2}}{\tau^2}-\di\bigg(\mathbb C(\chi^k)\e(u^k)
						+\mathbb D(\chi^k)\e\Big(\frac{u^k-u^{k-1}}{\tau}\Big)\bigg)=\ell^k
				\label{eqn:disc1}
					&&\text{a.e. in }\Omega,\\
					&\frac{\chi^k-\chi^{k-1}}{\tau}-\Delta\frac{\chi^k-\chi^{k-1}}{\tau}-\Delta\chi^k+\xi_\beta\Big(\frac{\chi^k-\chi^{k-1}}{\tau}\Big)\notag\\
				\label{eqn:disc2}
					&\qquad+\frac 12\big(\mathsf{c}_1'(\chi^k)+\mathsf{c}_2'(\chi^{k-1})\big)\CC\e(u^{k-1}):\e(u^{k-1})+f'(\chi^k)=0
					&&\text{a.e. in }\Omega,\\
				\label{eqn:disc4}
					&\Big(\mathbb C(\chi^k)\e(u^k)+\mathbb D(\chi^k)\e\Big(\frac{u^k-u^{k-1}}{\tau}\Big)\Big)\cdot \nu =b^k
					&&\text{a.e. on }\Gamma,\\
				\label{eqn:disc5}
					&\nabla\Big(\chi^k+\frac{\chik-\chi^{k-1}}{\tau}\Big)\cdot \nu =0
					&&\text{a.e. on }\Gamma
				\end{align}
				\end{subequations}
				for all $k=1,\ldots,M$,
				where $\mathsf{c}=\mathsf{c}_1+\mathsf{c}_2$ denotes the convex-concave decomposition from (A2).
		\end{enumerate}
	\end{definition}
	\begin{remark}
		If we assume $\nabla\chi^0\cdot \nu =0$ a.e. on $\Gamma$ we even obtain
		\begin{align}
			\label{eqn:cont5a}
			&\nabla\chi\cdot \nu =\nabla \chi_t\cdot \nu =0\quad\text{a.e. on }\Sigma
		\end{align}
		instead of \eqref{eqn:boundaryEq2} or \eqref{eqn:boundaryRegEq2} and for all $k=1,\ldots,M$
		\begin{align}
			\label{eqn:disc5a}
			&\nabla\chi^k\cdot \nu =\nabla\frac{\chi^k-\chi^{k-1}}{\tau}\cdot \nu =0\quad\text{a.e. on }\Gamma
		\end{align}
		instead of \eqref{eqn:disc5}.
	\end{remark}

\subsection{Existence of solutions}
\label{sec:existence}
\subsubsection{Existence for the time-discrete regularized system}
	At first we are going to show existence of time-discrete solution
	according to Definition \ref{def:notionSolution} (iii).
	Let $\tau>0$ and $\beta>0$.
	To enhance readability, we will mostly omit the subscripts $\tau$ and $\beta$ in
	$\ukb$ and $\chikb$.
	
	
	\begin{lemma}
	\label{lemma:existenceDiscr}
		Let the data $u^{0},v^0\in H^2(\Omega;\R^n)$, $\chi^0\in H^2(\Omega)$,
		$b^k\in H^{1/2}(\Gamma;\R^n)$ and $\ell^k\in L^2(\Omega;\R^n)$ for $k=0,\ldots, M$ be given.
		Then, there exists a strong solution $\{u^k,\chi^k\}_{k=0,\ldots,M}$ of the time-discrete system in the sense of Definition \ref{def:notionSolution} (iii).
	\end{lemma}
	\begin{proof}
		Starting from the initial values $(u^0,u^{-1},\chi^0)$ with $u^{-1}:=u^0-\tau v^0$ we are
		going to construct $\{u^k,\chi^k\}_{k=0,\ldots,M}$ by a recursive procedure.
		To this end, we decouple the discrete PDE problem into two distinct elliptic problems
		such that $\chik$ is obtained from $\chikk$ and $\ukk$,
		while $\uk$ is gained from $\ukk$, $\ukkk$, $\chik$, $b^k$ and $\ell^k$.
		\vspace*{0.5em}\\\textit{\underline{Step 1:} establishing equations \eqref{eqn:disc2}
		and \eqref{eqn:disc5}}\vspace*{0.5em}\\
		Let us define the functional $\C F:H^1(\Omega)\to\R$ by
		\begin{align*}
			\C F(\chi):={}&\int_\Omega\Big(\frac12|\nabla\chi|^2
				+\frac 12 \mathsf{c}_1(\chi)\CC\e(\ukk):\e(\ukk)
				+\frac 12 \mathsf{c}_2'(\chikk)\chi\CC\e(\ukk):\e(\ukk)\Big)\dx\\
			&+\int_\Omega\Big(f(\chi)+\tau I_\beta\big((\chi-\chikk)\tau^{-1}\big)\Big)\dx+\frac{\tau}{2}\int_\Omega\left|(\chi-\chikk)\tau^{-1}\right|^2\dx\\
			&+\frac{\tau}{2}\int_\Omega\left|\nabla (\chi-\chikk)\tau^{-1}\right|^2\dx
		\end{align*}
		By the direct method in the calculus of variations, we obtain the existence of a minimizer
		of $\C F$, which will be denoted by $\chi^k$.
		
		The Euler-Lagrange equation for the minimizer yields \eqref{eqn:disc2} in a weak form.
		By noticing that \eqref{eqn:disc2} is a elliptic equation for $\chik$ with right hand side in $L^2(\Omega)$,
		we conclude $\chik\in H^2(\Omega)$ by elliptic regularity results for Neumann problems
		(see, e.g. \cite[Theorem 2.4.2.7]{Gr85} and remember that $\Gamma$ is a $C^2$-boundary by Assumption (A1)).
		\vspace*{0.5em}\\\textit{\underline{Step 2:} establishing equations
		\eqref{eqn:disc1} and \eqref{eqn:disc4}}\vspace*{0.5em}\\
		Given the functions $\chik\in H^2(\Omega)$, $\ukk,\ukkk\in H^2(\Omega;\Rn)$, 
		$b^k\in H^{1/2}(\Gamma;\Rn)$ and $\ell^k\in L^2(\Omega;\Rn)$, we obtain a unique weak
		solution $\uk\in H^1(\Omega;\Rn)$ of the linear elliptic system \eqref{eqn:disc1}
		via the well-known Lax-Milgram theorem (remember the assumption $\DD=\mu\CC$ from (A3)):
		\begin{align}
			&\int_\Omega\Big(\big(\tau^2 \mathsf{c}(\chik)+\tau \mathsf{d}(\chik)\mu \big)\CC\e(u):\e(\zeta)+u\cdot\zeta\Big)\dx\notag\\
		\label{eqn:ellipticUSystem}
			&\qquad=\int_\Omega\bl\tau^2\ell^k-\tau\di\big( \mathsf{d}(\chik)\DD\e(\ukk)\big)+2\ukk-\ukkk\br\cdot\zeta\dx
				+\int_\Gamma b^k\cdot\zeta\dx
		\end{align}
		holding for all $\zeta\in H^1(\Omega;\R^n)$.
		
		Now we use a modification of the regularity argument in
		\cite[Proof of Lemma 4.1]{WIAS1927} and make use of the $C^2$-regularity of $\Gamma$ (see (A1)):\\
		If we consider the test-function $\zeta=\big(\tau^2 \mathsf{c}(\chik)+\tau \mathsf{d}(\chik)\mu \big)^{-1}\varphi$
		where $\varphi\in H^1(\Omega;\R^n)$ is another test-function (remember that $\mathsf{c}\geq 0$ and $\mathsf{d}\geq\eta>0$, see (A2)-(A3))
		the linear elliptic system \eqref{eqn:ellipticUSystem} rewrites as
		\begin{align}
		\label{eqn:elasticSystem}
			\mathbf a(u^k,\varphi)= \langle \mathbf q, \varphi \rangle_{H^1},\qquad \varphi\in H^{1}(\Omega;\Rn).
		\end{align}
		with the bilinear form
		$$
			\mathbf a(u,\varphi):=\int_\Omega\CC\e(u):\e(\varphi)\dx
		$$
		and the right hand side $\mathbf q\in H^1(\Omega;\Rn)'$ given by
		$$
			\langle \mathbf q, \varphi \rangle_{H^1}
				:=\int_\Omega R\cdot\varphi\dx+\int_\Gamma N\cdot\varphi\dx,
		$$
		where $R$ and $N$ are defined as
		\begin{align*}
			&R:=\frac{\tau^2 \mathsf{c}'(\chik)+\tau\mu \mathsf{d}'(\chik)}{\tau^2 \mathsf{c}(\chik)+\tau \mathsf{d}(\chik)\mu }\nabla\chik\cdot\CC\e(\uk)
				+\frac{\tau^2\ell^k-\tau\di\bl \mathsf{d}(\chi^k)\DD\e(\ukk)\br-\uk+2\ukk-\ukkk}{\tau^2 \mathsf{c}(\chik)+\tau \mathsf{d}(\chik)\mu},\\
			&N:=\frac{b^k}{\tau^2 \mathsf{c}(\chik)+\tau \mathsf{d}(\chik)\mu}.
		\end{align*}
		Note that $N\in H^{1/2}(\Gamma;\Rn)$ and $R\in L^{p}(\Omega;\Rn)$ for all $p\in(1,2)$,
		since $\e(\uk)\in L^2(\Omega;\Rnn)$ and $\nabla\chik\in L^q(\Omega;\Rn)$ for all $q\in[1,+\infty)$.
		
		In particular, $\mathbf q\in H^{2-s}(\Omega;\Rn)'$ for all $s\in (1,3/2)$.
		We gain $\uk\in H^s(\Omega;\Rn)$ by applying the lower Sobolev $H^{s}$-regularity result from
		\cite[Theorem 3.4.5 (ii)]{CDN10}.
		This, in turn, implies $\e(\uk)\in L^{2^*}(\Omega;\Rn)$  with the \textit{fractional critical exponent} given in this case by
		$2^*=\frac{2n}{n-(s-1)2}>2$ (see, e.g., \cite[Theorem 6.7]{DPV11}). We obtain $R\in L^2(\Omega;\Rn)$.
		The $H^{2}$-regularity result \cite[Theorem 3.4.1]{CDN10} applied to the linear elliptic system \eqref{eqn:elasticSystem} shows
		$\uk\in H^2(\Omega;\Rn)$.
		Thus \eqref{eqn:disc1} is shown.
		\ep
	\end{proof}
	
	

	
\subsubsection{Existence result for the time-continuous system}
	The aim of this section is to provide existence of strong solutions in the sense of
	Definition \ref{def:notionSolution} (i) and (ii).
	To this end, several a priori estimates for the time-discrete solutions will be established.
	The estimates will be used for the time-continuous limit analysis and for the optimal control problem in Section
	\ref{section:OCP}.
	
	We assume that the initial data $(u^0,v^0,\chi^0)$ satisfy
	\begin{subequations}
	\label{eqn:initialData}
	\begin{align}
		&u^0\in H^2(\Omega;\Rn),\\
		&v^0\in H^1(\Omega;\Rn),\\
		&\chi^0\in H_N^2(\Omega):=\big\{v\in H^2(\Omega)\,|\,\nabla v\cdot \nu=0\text{ a.e. on }\Gamma\big\}
	\end{align}
	\end{subequations}
	and the external forces $(b,\ell)$ are assumed to be in the following spaces:
	\begin{align}
		b\in L^{2}(0,T;H^{1/2}(\Gamma;\Rn))\cap H^{1}(0,T;L^{2}(\Gamma;\Rn)),\qquad\ell\in L^{2}(0,T;L^2(\Omega;\Rn)).
	\label{eqn:externalForces}
	\end{align}
	For the moment, let us consider some approximations
	\begin{align*}
		&\{v^0_\lambda\}_{\lambda\in(0,1)}\subseteq H^2(\Omega;\Rn),\\
		&\{b_\lambda\}_{\lambda\in(0,1)}\subseteq C^{1,1}(0,T;H^{1/2}(\Gamma;\Rn),\\
		&\{\ell_\lambda\}_{\lambda\in(0,1)}\subseteq C^{0,1}(0,T;L^2(\Omega;\Rn))
	\end{align*}
	of the the initial velocity $v^0$ and the external forces $b$ and $\ell$ such that (e.g. construction via convolution)
	\begin{subequations}
	\label{eqn:ellBConv}
	\begin{align}
		&v_\lambda^0\to v^0\;\,\text{strongly in }H^1(\Omega;\Rn),\\
		&b_\lambda\to b\quad\text{strongly in }L^{2}(0,T;H^{1/2}(\Gamma;\Rn))\cap H^{1}(0,T;L^{2}(\Gamma;\Rn)),\\
		&\ell_\lambda\to \ell\quad\text{strongly in }L^2(0,T;L^2(\Omega;\Rn))
	\end{align}
	\end{subequations}
	as $\lambda\downarrow 0$.
	Let us define the time-discretizations $b_{\tau,\lambda}^k$ and $\ell_{\tau,\lambda}^k$ by
	\begin{align*}
		b_{\tau,\lambda}^k:=b_\lambda(\tau k),
		\qquad\qquad
		\ell_{\tau,\lambda}^k:=\ell_\lambda(\tau k).
	\end{align*}
	For a sequence $\{h^k\}_{k=0,\ldots,M}$ where
	$h^k\in\{u_{\tau,\beta}^k,\chi_{\tau,\beta}^k,b_{\tau,\lambda}^k,\ell_{\tau,\lambda}^k\}$, we define
	the piecewise constant and linear interpolation as
	\begin{align}
	\begin{split}
	\label{eqn:interpolation}
		\left.
		\begin{matrix}
		&\displaystyle\ol{h}(t):=h^k,\qquad \ul{h}(t):=h^{k-1},\qquad \ul{\ul{h}}(t):=h^{k-2},\\\\
		&\displaystyle h(t):=\frac{t-(k-1)\tau}{\tau}h^k+\frac{k\tau-t}{\tau}h^{k-1}
		\end{matrix}
		\right\}
		\qquad\text{for $t\in((k-1)\tau,k\tau]$.}
	\end{split}
	\end{align}
	The left-continuous and right-continuous piecewise constant interpolation for a
	given time point $t$ is denoted by
	\begin{align*}
		&\ol t_\tau:=\tau k\quad\text{ for }\tau(k-1)<t\leq\tau k,\\
		&\ul t_\tau:=\tau k\quad\text{ for }\tau k\leq t<\tau(k+1).
	\end{align*}
	For notational convenience, we define
	the time-discrete velocity field and their interpolations by
	\begin{align}
		v_{\tau,\beta}^k:=\frac{u_{\tau,\beta}^k-u_{\tau,\beta}^{k-1}}{\tau}\text{ for }k=0,\ldots,M\qquad\text{and}\qquad
		\ol v_{\tau,\beta},\ul v_{\tau,\beta},v_{\tau,\beta}\text{ by \eqref{eqn:interpolation}}.
	\label{eqn:velocities}
	\end{align}
	As a first result, we prove convergence of the discretizations of the given data.
	\begin{lemma}
	\label{lemma:dataConv}
		There exist subsequences $\tau_k\downarrow0$ and $\lambda_k\downarrow0$ as $k\uparrow \infty$ such that
		\begin{align*}
			v_{\tau_k}^0&\to v^0\;\;\text{strongly in }H^{1}(\Omega;\Rn),\\
			b_{\tau_k,\lambda_k}&\to b\quad\text{strongly in }L^{2}(0,T;H^{1/2}(\Gamma;\Rn))\cap H^{1}(0,T;L^{2}(\Gamma;\Rn)),\\
			\ol\ell_{\tau_k,\lambda_k}&\to \ell\quad\text{strongly in }L^2(0,T;L^2(\Omega;\Rn))
		\end{align*}
		as $k\uparrow\infty$.
		For readers' convenience we set $b_{\tau_k}:=b_{\tau_k,\lambda_k}$ and
		$\ol\ell_{\tau_k}:=\ol\ell_{\tau_k,\lambda_k}$ and
		omit the subscript $k$.
		Then the statement above reads as		
		$v_\tau^0\to v^0$,  $b_{\tau}\to b$ and $\ol\ell_{\tau}\to \ell$ as $\tau\downarrow 0$.
	\end{lemma}
	\begin{proof}
		For every fixed $\lambda>0$, we find
		\begin{subequations}
		\label{eqn:ellBConv2}
		\begin{align}
		\label{eqn:bConv2}
			&b_{\tau,\lambda}\to b_\lambda\quad\text{strongly in }L^{2}(0,T;H^{1/2}(\Gamma;\Rn))\cap H^{1}(0,T;L^{2}(\Gamma;\Rn)),\\
		\label{eqn:ellConv2}
			&\ol\ell_{\tau,\lambda}\to \ell_\lambda\quad\text{strongly in }L^2(0,T;L^2(\Omega;\Rn))
		\end{align}
		\end{subequations}
		as $\tau\downarrow 0$.
		Indeed, the first convergence in \eqref{eqn:bConv2} follows by exploiting the Lipschitz continuity
		of $b_\lambda\in C^{0,1}(0,T;H^{1/2}(\Gamma;\Rn))$.
		Property \eqref{eqn:ellConv2} can be proven with a similar argument.
		The convergence $b_{\tau,\lambda}\to b_\lambda$ in the $H^{1}(0,T;L^{2}(\Gamma;\Rn))$-norm follows by the fundamental theorem of calculus for $X$-valued
		functions where $X:=L^{2}(\Gamma;\Rn)$
		and by the Lipschitz continuity of $\partial_t b_\lambda\in C^{0,1}(0,T;X)$:
		\begin{align*}
			\int_0^T\Big\|\partial_t b_{\tau,\lambda}(t)-\partial_t b_\lambda(t)\Big\|_{X}^2\dt
			={}&\int_0^T\Big\|\frac{b_{\lambda}(\ol t_\tau)-b_{\lambda}(\ul t_\tau)}{\tau}-\partial_t b_\lambda(t)\Big\|_{X}^2\dt\\
			={}&\int_0^T\Big\|\frac{1}{\tau}\int_{\ul t_\tau}^{\ol t_\tau}\big(\partial_t b_{\lambda}(s)-\partial_t b_\lambda(t)\big)\ds\Big\|_{X}^2\dt\\
			\leq{}&\int_0^T\Big(\frac{1}{\tau}\int_{\ul t_\tau}^{\ol t_\tau}\big\|\partial_t b_{\lambda}(s)-\partial_t b_\lambda(t)\big\|_X\ds\Big)^2\dt\\
			\leq{}&C\int_0^T\Big(\frac{1}{\tau}\int_{\ul t_\tau}^{\ol t_\tau}|s-t|\ds\Big)^2\dt\\
			\leq{}&CT\tau^2.
		\end{align*}
		The claim follows by using the convergences \eqref{eqn:ellBConv},
		\eqref{eqn:ellBConv2} and the following ``$\varepsilon/2$''-argument:
		
		For a given $\varepsilon>0$ we may choose a small $\lambda>0$ such that
		$\|b-b_\lambda\|<\varepsilon/2$ due to \eqref{eqn:ellBConv}.
		For every such $\lambda$ we choose a small $\tau>0$ such that
		$\|b_\lambda-b_{\tau,\lambda}\|<\varepsilon/2$ due to \eqref{eqn:ellBConv2}.
		In consequence we find for every $\varepsilon>0$ small values
		$\lambda>0$ and $\tau>0$ such that
		$\|b-b_{\tau,\lambda}\|<\varepsilon$.
		\ep
	\end{proof}\\
	\begin{lemma}[A~priori estimates for the time-discrete system]
	\label{lemma:aPrioriDiscr}
	The following a~priori estimates hold for strong solutions of the time-discrete system given in Definition \ref{def:notionSolution} (iii)
	(recall that \eqref{eqn:velocities} implies $\ol v_{\tau,\beta}=\partial_t u_{\tau,\beta}$ ):
	\begin{enumerate}
		\item[(i)]First a priori estimate:\\
			There exists a constant $C>0$ which continuously depends on
			\begin{align}
				C=C\big(\|u^0\|_{H^1},\|v^0\|_{L^2},\|\chi^0\|_{H^1},\|b\|_{L^{2}(0,T;L^{2}(\Gamma;\Rn))},\|\ell\|_{L^2(0,T;L^2)}\big)
			\label{eqn:aprioriC}
			\end{align}
			such that for all $\tau,\beta>0$
			\begin{align*}
				&		\|u_{\tau,\beta}\|_{H^1(0,T;H^1)\cap W^{1,\infty}(0,T;L^2)}\leq C,
				&&	\|\chi_{\tau,\beta}\|_{H^1(0,T;H^{1})}\leq C,\\
				&		\|\ul u_{\tau,\beta}\|_{L^\infty(0,T;H^1)}\leq C,
				&&	\|\ul\chi_{\tau,\beta}\|_{L^\infty(0,T;H^{1})}\leq C,\\
				&		\|\ol u_{\tau,\beta}\|_{L^\infty(0,T;H^1)}\leq C,
				&&	\|\ol\chi_{\tau,\beta}\|_{L^\infty(0,T;H^{1})}\leq C.
			\end{align*}
		\item[(ii)]Second a priori estimate:\\
			There exists a constant $D>0$ which continuously depends on
			$$
				D=D\big(\|u^0\|_{H^2},\|v^0\|_{H^1},\|\chi^0\|_{H^2}, \|b\|_{L^{2}(0,T;H^{1/2}(\Gamma;\Rn))\cap H^{1}(0,T;L^{2}(\Gamma;\Rn))},\|\ell\|_{L^2(0,T;L^2)}\big)
			$$
			such that for all $\tau,\beta>0$
			\begin{align*}
				&		\|u_{\tau,\beta}\|_{H^1(0,T;H^{2})\cap W^{1,\infty}(0,T;H^1)}\leq D,
				&&	\|\chi_{\tau,\beta}\|_{H^1(0,T;H^{2})}\leq D,\\
				&		\|\ul{u}_{\tau,\beta}\|_{L^\infty(0,T;H^{2})}\leq D,
				&&	\|\ul \chi_{\tau,\beta}\|_{L^\infty(0,T;H^{2})}\leq D,\\
				&		\|\ol{u}_{\tau,\beta}\|_{L^\infty(0,T;H^{2})}\leq D,
				&&	\|\ol \chi_{\tau,\beta}\|_{L^\infty(0,T;H^{2})}\leq D,\\
				&		\|v_{\tau,\beta}\|_{L^2(0,T;H^2)\cap L^{\infty}(0,T;H^1)\cap H^1(0,T;L^{2})}\leq D,
				&&	\|\xi_\beta(\partial_t\chi_{\tau,\beta})\|_{L^2(0,T;L^2)}\leq D.
			\end{align*}
	\end{enumerate}
	\end{lemma}
	\textbf{Proof.} 
	We will omit the subscript $\tau$ and $\beta$ in the time-discrete solutions.
	\begin{enumerate}
			\item[]\hspace*{-1em}\underline{To (i):}
			In the following, we make use of a combined convex-concave estimate for:
			A convexity estimate for $\mathsf{c}_1$ and concavity estimate for $\mathsf{c}_2$ yield:
			\begin{align*}
				&\mathsf{c}_1(\chikk)-\mathsf{c}_1(\chik)\geq \mathsf{c}_1'(\chik)(\chikk-\chik)\\
				&\mathsf{c}_2(\chikk)-\mathsf{c}_2(\chik)\geq \mathsf{c}_2'(\chikk)(\chikk-\chik).
			\end{align*}
			Adding them shows
			\begin{align*}
				\mathsf{c}(\chikk)-\mathsf{c}(\chik)
				&=\big(\mathsf{c}_1(\chikk)-\mathsf{c}_1(\chik)\big)
				+\big(\mathsf{c}_2(\chikk)-\mathsf{c}_2(\chik)\big)\\
				&\geq(\mathsf{c}_1'(\chik)+\mathsf{c}_2'(\chikk))(\chikk-\chik).
			\end{align*}
			By using this combined estimate and the positivity of $\CC$, it holds
			\begin{align}
				&\mathsf{c}(\chi^k)\CC\e(u^k):\e(u^k-u^{k-1})\notag\\
				&\qquad= \frac12 \mathsf{c}(\chi^k)\Big(\CC\e(\uk):\e(\uk)-\CC\e(\ukk):\e(\ukk)+\CC\e(\uk-\ukk):\e(\uk-\ukk)\Big)\notag\\
				&\qquad\geq \frac{\mathsf{c}(\chi^k)}{2}\CC\e(\uk):\e(\uk)
					-\frac{\mathsf{c}(\chi^{k-1})}{2}\CC\e(u^{k-1}):\e(\ukk)\notag\\
				&\qquad\quad+\frac12\big(\mathsf{c}(\chi^{k-1})-\mathsf{c}(\chi^k)\big)\CC\e(u^{k-1}):\e(\ukk)\notag\\
				&\qquad\geq\frac{\mathsf{c}(\chi^k)}{2}\CC\e(u^k):\e(\ukk)-\frac{\mathsf{c}(\chi^{k-1})}{2}\CC\e(\ukk):\e(\ukk)\notag\\
			\label{eqn:convConc}
				&\qquad\quad+\frac12\big(\mathsf{c}_1'(\chi^{k})+\mathsf{c}_2'(\chi^{k-1})\big)(\chi^{k-1}-\chi^k)\CC\e(u^{k-1}):\e(\ukk).
			\end{align}
			
			Now, by testing equation \eqref{eqn:disc1} with $\uk-\ukk$,
			integrating over $\Omega$, summing over the time index $k=1,\ldots,\ol t_\tau/\tau$,
			integrating by parts and using \eqref{eqn:disc4}, we obtain
			(remember that $v^k=(u^k-u^{k-1})/\tau$)
			\begin{align*}
					&\sum_{k=1}^{\ol t_\tau/\tau}\int_\Omega(v^k-v^{k-1})v^k\dx
					+\sum_{k=1}^{\ol t_\tau/\tau}\int_\Omega \mathsf{c}(\chik)\CC\e(\uk):\e(u^k-u^{k-1})\dx\\
					&+\int_0^{\ol t_\tau}\int_\Omega \mathsf{d}(\chik)\DD\e(\ol v):\e(\ol v)\dxs\\
					&\qquad=\int_0^{\ol t_\tau}\int_\Omega\ol\ell\cdot\ol v\dxs
						+\int_0^{\ol t_\tau}\int_\Gamma\ol b\cdot\ol v\dxs.
			\end{align*}
			Applying elementary estimates including
			the convex-concave estimate \eqref{eqn:convConc}, Korn's and Young's inequality and the
			trace theorem
			$H^1(\Omega;\Rn)\hookrightarrow L^2(\Gamma;\Rn)$ yield ($\eta,\delta, C_\delta>0$ are constants)
			\begin{align}
				&\frac 12\|\ol v(t)\|_{L^2}^2-\frac 12\|v^0\|_{L^2}^2
					+\underbrace{\int_\Omega\frac{\mathsf{c}(\ol\chi(t))}{2}\CC\e(\ol u(t)):\e(\ol u(t))\dx}_{
						\geq 0\text { by using (A2)}}
					-\int_\Omega\frac{\mathsf{c}(\chi^0)}{2}\CC\e(u^0):\e(u^0)\dx\notag\\
					&+\int_0^{\ol t_\tau}\int_\Omega\frac12\big(\mathsf{c}_1'(\ol\chi)+\mathsf{c}_2'(\ul\chi)\big)(-\partial_t\chi)\CC\e(\ul u):\e(\ul u)\dxs
					+\eta\|\e(\ol v)\|_{L^2(0,\ol t_\tau;L^2)}^2\notag\\
					&\qquad\leq C_\delta\|\ol\ell\|_{L^2(0,\ol t_\tau;L^2)}^2+C_\delta\|\ol b\|_{L^2(0,\ol t_\tau;L^2(\Gamma;\Rn))}^2+\delta\|\ol v\|_{L^2(0,\ol t_\tau;H^1)}^2.
			\label{eqn:aprioriU1}
			\end{align}
			
			Testing equation \eqref{eqn:disc2} with $\chik-\chikk$,
			integrating over $\Omega$, summing over the time index $k=1,\ldots,\ol t_\tau/\tau$, integrating by parts and using \eqref{eqn:disc5}, we obtain
			\begin{align*}
					&\|\partial_t\chi\|_{L^2(0,\ol t_\tau;L^2)}^2
					+\|\nabla\partial_t\chi\|_{L^2(0,\ol t_\tau;L^2)}^2
					+\sum_{k=1}^{\ol t_\tau/\tau}\int_\Omega\nabla\chik\cdot\nabla(\chik-\chikk)\dx\notag\\
					&+\sum_{k=1}^{\ol t_\tau/\tau}\int_\Omega\xi_\beta\Big(\frac{\chik-\chikk}{\tau}\Big)(\chik-\chikk)\dx
					+\int_0^{\ol t_\tau}\int_\Omega\frac12\big(\mathsf{c}_1'(\ol\chi)+\mathsf{c}_2'(\ul\chi)\big)\partial_t\chi\CC\e(\ul u):\e(\ul u)\dxs\\
					&+\int_0^{\ol t_\tau}\int_\Omega f'(\ol\chi)\partial_t\chi\dxs\\
					&\qquad=0
			\end{align*}
			By using the monotonicity of $\xi_\beta$ (see Definition \ref{def:regularization}), we get $\xi_\beta\Big(\frac{\chik-\chikk}{\tau}\Big)(\chik-\chikk)\geq 0$.
			Together with elementary convexity estimates, the Lipschitz continuity of $f'$ (see (A4))
			and Young's inequality, we find
			\begin{align}
					&\|\partial_t\chi\|_{L^2(0,\ol t_\tau;L^2)}^2
					+\|\nabla\partial_t\chi\|_{L^2(0,\ol t_\tau;L^2)}^2
					+\frac 12\|\nabla \ol\chi(t)\|_{L^2}^2-\frac 12\|\nabla \chi^0\|_{L^2}^2\notag\\
					&\quad+\int_0^{\ol t_\tau}\int_\Omega\frac12\big(\mathsf{c}_1'(\ol\chi)+\mathsf{c}_2'(\ul\chi)\big)\partial_t\chi\CC\e(\ul u):\e(\ul u)\dxs\notag\\
					&\qquad\leq C_\delta(\|\ol\chi\|_{L^2(0,\ol t_\tau;L^2)}^2+1)+\delta\|\partial_t\chi\|_{L^2(0,\ol t_\tau;L^2)}^2
			\label{eqn:aprioriChi1a}
			\end{align}
			To proceed, we consider the calculation
			\begin{align}
				\frac12\|\ol\chi(t)\|_{L^2}^2
				={}&\int_0^{\ol t_\tau}\int_\Omega\frac{\mathrm{d}}{\dt}\frac12|\chi|^2\dxs
					-\frac12\|\chi^0\|_{L^2}^2\notag\\
				={}&\int_0^{\ol t_\tau}\int_\Omega\chi\partial_t\chi\dxs-\frac12\|\chi^0\|_{L^2}^2\notag\\
				\leq{}&\delta\|\partial_t\chi\|_{L^2(0,\ol t_\tau;L^2)}^2
					+C_\delta\|\chi\|_{L^2(0,\ol t_\tau;L^2)}^2-\frac12\|\chi^0\|_{L^2}^2.
			\label{eqn:estTrick}
			\end{align}
			Adding $\frac12\|\ol\chi(t)\|_{L^2}^2$ on both sides in \eqref{eqn:aprioriChi1a}
			and using \eqref{eqn:estTrick} on the right-hand side,
			we find
			\begin{align}
					&\Big(\frac12-\delta\Big)\|\partial_t\chi\|_{L^2(0,\ol t_\tau;L^2)}^2
					+\|\nabla\partial_t\chi\|_{L^2(0,\ol t_\tau;L^2)}^2
					+\frac 12\|\ol\chi(t)\|_{H^1}^2-\frac 12\|\chi^0\|_{H^1}^2\notag\\
					&\quad+\int_0^{\ol t_\tau}\int_\Omega\frac12\big(\mathsf{c}_1'(\ol\chi)+\mathsf{c}_2'(\ul\chi)\big)\partial_t\chi\CC\e(\ul u):\e(\ul u)\dxs\notag\\
					&\qquad\leq C_\delta(\|\ol\chi\|_{L^2(0,\ol t_\tau;L^2)}^2+1).
			\label{eqn:aprioriChi1}
			\end{align}
			
			Adding \eqref{eqn:aprioriChi1} and \eqref{eqn:aprioriU1},
			and choosing $\delta>0$ small, 
			we see that the term
			$$
				\int_0^{\ol t_\tau}\int_\Omega\frac12\big(\mathsf{c}_1'(\ol\chi)+\mathsf{c}_2'(\ul\chi)\big)\partial_t\chi\CC\e(\ul u):\e(\ul u)\dxs
			$$
			cancels out in the calculations and we obtain
			\begin{align}
				&\|\ol v(t)\|_{L^2}^2
					+\|\ol\chi(t)\|_{H^1}^2
					+\|\e(\ol v)\|_{L^2(0,\ol t_\tau;L^2)}^2
					+\|\partial_t\chi\|_{L^2(0,\ol t_\tau;L^2)}^2
					+\|\nabla\partial_t\chi\|_{L^2(0,\ol t_\tau;L^2)}^2\notag\\
					&\leq C\Big(
						1
						+\|u^0\|_{H^1}^2
						+\|v^0\|_{L^2}^2
						+\|\chi^0\|_{H^1}^2
						+\|\ol\ell\|_{L^2(0,\ol t_\tau;L^2)}^2
						+\|\ol b\|_{L^2(0,\ol t_\tau;L^2(\Gamma;\Rn))}^2
						+\|f'(\ol\chi)\|_{L^2(0,\ol t_\tau;L^2)}^2
						\Big)\notag\\
			\label{eqn:est2}
						&\quad+C\int_0^{\ol t_\tau}\|\ol\chi\|_{L^2}^2\ds.
			\end{align}
			Korn's inequality yields
			\begin{align*}
				\|\e(\ol v)\|_{L^2(0,\ol t_\tau;L^2)}^2\geq \frac{1}{C}\|\ol v\|_{L^2(0,\ol t_\tau;H^1)}^2-\|\ol v\|_{L^2(0,\ol t_\tau;L^2)}^2.
			\end{align*}
			We thus obtain from \eqref{eqn:est2}
			\begin{align*}
				&\|\ol v(t)\|_{L^2}^2
					+\|\ol\chi(t)\|_{H^1}^2
					+\|\ol v\|_{L^2(0,\ol t_\tau;H^1)}^2
					+\|\partial_t\chi\|_{L^2(0,\ol t_\tau;H^1)}^2\notag\\
					&\qquad\leq C\Big(
						1
						+\|u^0\|_{H^1}^2
						+\|v^0\|_{L^2}^2
						+\|\chi^0\|_{H^1}^2
						+\|\ol\ell\|_{L^2(0,\ol t_\tau;L^2)}^2
						+\|\ol b\|_{L^2(0,\ol t_\tau;L^2(\Gamma;\Rn))}^2
						\Big)\\
					&\qquad\quad+C\int_0^{\ol t_\tau}\Big(\|\ol v\|_{L^2}^2+\|\ol \chi\|_{L^2}^2\Big)\ds.
			\end{align*}
			We end up with the desired estimates in (i) by using the discrete version of Gronwall's lemma
			and
			$$
				\|\ol u(t)\|_{H^1}^2=\Big\|u^0+\int_0^{\ol t_\tau}\ol v\ds\Big\|_{H^1}^2
				\leq C\big(\|u^0\|_{H^1}^2+\|\ol v\|_{L^2(0,\ol t_\tau;H^1)}^2\big)
			$$
			afterwards.
			
		\item[]\hspace*{-1em}\underline{To (ii) -- local-in-time estimate:}\\
			At first we are going to show the a priori estimates in (ii) for small time.
			In the next step global-in-time estimates will be derived.
			
			Testing equation \eqref{eqn:disc1} with $-\tau\di\Big(\mathsf{c}(\chik)\CC\e(\uk)+\mathsf{d}(\chik)\DD\e\Big(\frac{\uk-\ukk}{\tau}\Big)\Big)$, integrating over $\Omega$ in space
			and summing over the time index $k=1,\ldots,\ol t_\tau/\tau$, we may write the result in the
			following way
			\begin{align}
			 &\underbrace{\int_0^{\ol t_\tau}\int_\Omega-\partial_t v\cdot\di\big(\mathsf{c}(\ol\chi)\CC\e(\ol u)+\mathsf{d}(\ol\chi)\DD\e(\ol v)\big)\dxs}_{=:T_1}\notag\\
			 &+\int_0^{\ol t_\tau}\int_\Omega\frac12\big|\di\big(\mathsf{c}(\ol\chi)\CC\e(\ol u)+\mathsf{d}(\ol\chi)\DD\e(\ol v)\big)\big|^2\dxs\notag\\
			 &+\underbrace{
			 \int_0^{\ol t_\tau}\int_\Omega\frac12\Big|\mathsf{c}'(\ol\chi)\nabla \ol\chi\cdot\CC\e(\ol u)+\mathsf{c}(\ol\chi)\di\big(\CC\e(\ol u)\big)
			 	+\mathsf{d}'(\ol\chi)\nabla \ol\chi\cdot\DD\e(\ol v)+\mathsf{d}(\ol\chi)\di\big(\DD\e(\ol v)\big)\Big|^2\dxs}_{=:T_2}\notag\\
			 &\qquad
					=\underbrace{\int_0^{\ol t_\tau}\int_\Omega -\ol\ell\cdot\di\big(\mathsf{c}(\ol\chi)\CC\e(\ol u)+\mathsf{d}(\ol\chi)\DD\e(\ol v)\big)\dxs}_{
						\leq C_\delta\|\ol\ell\|_{L^2(L^2)}^2+\delta\|\di(\mathsf{c}(\ol\chi)\CC\e(\ol u)+\mathsf{d}(\ol\chi)\DD\e(\ol v))\|_{L^2(L^2)}^2}.
			\label{eqn:est6}
			\end{align}
			Note that the second summand and the third summand, i.e. $T_2$, are identical.
			The splitting will simplify the calculations.
			
			Testing equation \eqref{eqn:disc2} with $-\Delta(\chik-\chikk)$, integrating over $\Omega$ in space
			and summing over $k=1,\ldots,\ol t_\tau/\tau$, we obtain
			\begin{align}
				&\underbrace{-\int_0^{\ol t_\tau}\int_\Omega\partial_t\chi\Delta\partial_t\chi\dxs}_{=\|\nabla\partial_t\chi\|_{L^2(L^2)}^2\text{ by using \eqref{eqn:disc5a}}}
					+\|\Delta\partial_t\chi\|_{L^2(0,\ol t_\tau;L^2)}^2
					+\underbrace{\sum_{k=1}^{\ol t_\tau/\tau}\int_\Omega\Delta\chik(\Delta\chik-\Delta\chikk)\dx}_{\geq \frac12\|\Delta\ol\chi(t)\|_{L^2}^2-\frac12\|\Delta\chi^0\|_{L^2}^2}
					\notag\\
				&\underbrace{-\int_0^{\ol t_\tau}\int_\Omega\xi_\beta(\partial_t\chi)\Delta\partial_t\chi\dxs}_{=:T_3}
					\underbrace{-\int_0^{\ol t_\tau}\int_\Omega\frac 12(\mathsf{c}_1'(\ol\chi)+\mathsf{c}_2'(\ul\chi))\CC\e(\ul u):\e(\ul u)\Delta\partial_t\chi\dxs}_{=:T_4}\notag\\
				&\underbrace{-\int_0^{\ol t_\tau}\int_\Omega f'(\ol\chi)\Delta\partial_t\chi\dxs}_{=:T_5}\notag\\
				&\qquad=0.
			\label{eqn:est4}
			\end{align}
			In the following, we are going to estimate $T_1,\ldots,T_{7}$ and conclude the claimed a priori estimates thereafter:
			\begin{itemize}
				\item[]\hspace*{-1.5em}-- To $(T_1)$:
				Integration by parts in space yields
				\begin{align*}
					T_1
					={}&\underbrace{\int_0^{\ol t_\tau}\int_\Omega\e(\partial_t v):\mathsf{c}(\ol\chi)\CC\e(\ol u)\dxs}_{=:T_1^{(1)}}
						+\underbrace{\int_0^{\ol t_\tau}\int_\Omega\e(\partial_t v):\mathsf{d}(\ol\chi)\DD\e(\ol v)\dxs}_{=:T_1^{(2)}}\\
					&\underbrace{-\int_{0}^{\ol t_\tau}\int_\Gamma\partial_t v\cdot\bl\big(\mathsf{c}(\ol\chi)\CC\e(\ol u)+\mathsf{d}(\ol\chi)\DD\e(\ol v)\big)\cdot\nu \br\dxs}_{=:T_1^{(3)}}
				\end{align*}
				Note that we have no compensating $\partial_t v\,$-term on the left-hand side of \eqref{eqn:est6}.
				To circumvent this problem we rewrite the term $T_1^{(1)}$ by using the discrete integration by parts formula in time
				$$
					\sum_{k=1}^N \tau\frac{a^k-a^{k-1}}{\tau} b^k
						=a^N b^N-a^0 b^0-\sum_{k=1}^N \tau a^{k-1}\frac{b^k-b^{k-1}}{\tau}.
				$$
				Together with the boundedness of $\mathsf{c}$ and $\mathsf{c}'$ (see (A2)) we find:
				\begin{align*}
					T_1^{(1)}={}&-\int_{0}^{\ol t_\tau}\int_\Omega \e(\ul v):\frac{\mathsf{c}(\ol\chi)\CC\e(\ol u)-\mathsf{c}(\ul\chi)\CC\e(\ul u)}{\tau}\dxs\\
						&+\int_\Omega \e(\ol v(t)):\mathsf{c}(\ol\chi(t))\CC\e(\ol u(t))\dx-\int_\Omega \e(v^0):\mathsf{c}(\chi^0)\CC\e(u^0)\dx\\
					={}&-\int_{0}^{\ol t_\tau}\int_\Omega \e(\ul v):\frac{\mathsf{c}(\ol\chi)-\mathsf{c}(\ul\chi)}{\tau}\CC\e(\ol u)\dxs
						-\int_{0}^{\ol t_\tau}\int_\Omega \e(\ul v):\mathsf{c}(\ul\chi)\CC\e(\ol v)\dxs\\
						&+\int_\Omega \e(\ol v(t)):\mathsf{c}(\ol\chi(t))\CC\e(\ol u(t))\dx
						-\int_\Omega \e(v^0):\mathsf{c}(\chi^0)\CC\e(u^0)\dx\\
					\geq{}&
						-C\|\mathsf{c}'\|_{L^\infty}\underbrace{\int_{0}^{\ol t_\tau}\int_\Omega|\e(\ul v)||\partial_t\chi||\e(\ol u)|\dxs}_{=:T_1^{(1,1)}}
						-C\underbrace{\int_{0}^{\ol t_\tau}\int_\Omega|\e(\ul v)||\e(\ol v)|\dxs}_{=:T_1^{(1,2)}}
						\\
						&-C\|\mathsf{c}\|_{L^\infty}\underbrace{\int_\Omega|\e(\ol v(t))||\e(\ol u(t))|\dx}_{=:T_1^{(1,3)}}
						-C\|\mathsf{c}\|_{L^\infty}\|\e(v^0)\|_{L^2}\|\e(u^0)\|_{L^4}.
				\end{align*}
				By using H\"older's and Young's inequalities, uniform boundedness of
				$\|\partial_t\chi\|_{L^2(0,T;L^4)}$,
				$\|\e(\ol u)\|_{L^\infty(0,T;L^2)}$,
				$\|\e(\ol v)\|_{L^2(0,T;L^2)}$,
				$\|\e(\ul v)\|_{L^2(0,T;L^2)}$
				and $\|\e(\ol v)\|_{L^2(0,T;L^2)}$ (see First a priori estimates), we obtain
				\begin{align*}
					T_1^{(1,1)}\leq{}&\|\partial_t\chi\|_{L^2(0,\ol t_\tau;L^4)}\|\e(\ul v)\|_{L^2(0,\ol t_\tau;L^4)}\|\e(\ol u)\|_{L^\infty(0,\ol t_\tau;L^2)}\\
						\leq{}&C_\delta+\delta\|\e(\ul v)\|_{L^2(0,\ol t_\tau;L^4)}^2,\\
					T_1^{(1,2)}\leq{}&\frac12\|\e(\ul v)\|_{L^2(0,T;L^2)}^2+\frac12\|\e(\ol v)\|_{L^2(0,T;L^2)}^2\leq C,\\
					T_1^{(1,3)}\leq{}&C_\delta\|\e(\ol u)\|_{L^\infty(0,T;L^2)}^2+\delta\|\e(\ol v(t))\|_{L^2}^2
						\leq C_\delta+\delta\|\e(\ol v(t))\|_{L^2}^2.
				\end{align*}
				The term $T_1^{(2)}$ can be estimated as follows:
				\begin{align*}
					T_1^{(2)}&=\int_0^{\ol t_\tau}\int_\Omega\frac{1}{2\tau}\bl \mathsf{d}(\ol\chi)\DD\e(\ol v):\e(\ol v)-\mathsf{d}(\ol\chi)\DD\e(\ul v):\e(\ul v)\br\dxs\notag\\
						&\quad
							+\int_0^{\ol t_\tau}\int_\Omega\frac{1}{2\tau}\mathsf{d}(\ol\chi)\DD\e(\ol v-\ul v):\e(\ol v-\ul v)\dxs\notag\\
						&=\underbrace{\int_0^{\ol t_\tau}\int_\Omega\frac{1}{2\tau}\bl \mathsf{d}(\ol\chi)\DD\e(\ol v):\e(\ol v)-\mathsf{d}(\ul\chi)\DD\e(\ul v):\e(\ul v)\br\dxs}_{=:T_1^{(2,1)}}\notag\\
							&\quad
							+\underbrace{\int_0^{\ol t_\tau}\int_\Omega\frac{1}{2\tau}\bl \mathsf{d}(\ul\chi)-\mathsf{d}(\ol\chi)\br\DD\e(\ul v):\e(\ul v)\dxs}_{=:T_1^{(2,2)}}\\
							&\quad
							+\underbrace{\int_0^{\ol t_\tau}\int_\Omega\frac{1}{2\tau}\mathsf{d}(\ol\chi)\DD\e(\ol v-\ul v):\e(\ol v-\ul v)\dxs}_{\geq 0}.
				\end{align*}
				For further estimations we make use of the
				\textit{Ladyzhenskaya's inequality} (see \cite{Lad58})
				\begin{align}
					\|w\|_{L^4}\leq C\|w\|_{H^1}^{1/2}\|w\|_{L^2}^{1/2}\quad \text{ valid for all }w\in H^1(\Omega),
					\label{eqn:GN}
				\end{align}
				which is a special version of
				\textit{Gagliardo-Nirenberg inequality} in 2D (see \cite{Nir59}).
				This inequality naturally generalizes to $\R^m$-valued Sobolev functions.

				By using \eqref{eqn:GN}, the property $\mathsf{d}(\cdot)\geq\eta>0$ and the Lipschitz continuity
				of $\mathsf{d}$ (see (A3)), we obtain
				\begin{align*}
					T_1^{(2,1)}&=\sum_{k=1}^{\ol t_\tau/\tau}\int_\Omega\frac12\bl \mathsf{d}(\chi^k)\DD\e(\vk):\e(\vk)-\mathsf{d}(\chi^{k-1})\DD\e(\vkk):\e(\vkk)\br\dx\\
						&=\int_\Omega\frac12\Big(\mathsf{d}(\ol\chi)\DD\e(\ol v(t)):\e(\ol v(t))-\mathsf{d}(\chi^0)\DD\e(v^0):\e(v^0)\Big)\dx\\
						&\geq\int_\Omega\bl\frac\eta2\e(\ol v(t)):\e(\ol v(t))-\frac12\mathsf{d}(\chi^0)\DD\e(v^0):\e(v^0)\br\dx,\\
					T_1^{(2,2)}&\geq -C\int_0^{\ol t_\tau}\int_\Omega|\partial_t\chi||\e(\ul v)|^2\dxs\\
						&\geq -C\int_0^{\ol t_\tau}\|\partial_t\chi\|_{L^4}\|\e(\ul v)\|_{L^2}\|\e(\ul v)\|_{L^4}\ds\\
						&\geq -\delta_1\int_0^{\ol t_\tau}\|\partial_t\chi\|_{L^4}^2\|\e(\ul v)\|_{L^2}^2\ds-C_{\delta_1}\int_0^{\ol t_\tau}\|\e(\ul v)\|_{L^4}^2\ds\\
						&\geq -\delta_1\int_0^{\ol t_\tau}\|\partial_t\chi\|_{H^1}^2\|\e(\ul v)\|_{L^2}^2\ds-C_{\delta_1}\int_0^{\ol t_\tau}\|\e(\ul v)\|_{L^2}\|\e(\ul v)\|_{H^1}\ds\\
						&\geq -\delta_1\int_0^{\ol t_\tau}\|\partial_t\chi\|_{H^1}^2\|\e(\ul v)\|_{L^2}^2\ds-C_{\delta_1}C_{\delta_2}\|\e(\ul v)\|_{L^2(0,\ol t_\tau;L^2)}^2
							-C_{\delta_1}\delta_2\|\e(\ul v)\|_{L^2(0,\ol t_\tau;H^1)}^2.
				\end{align*}
				Note that by choosing $\delta_1=\delta$ and $\delta_2=\delta C_{\delta_1}^{-1}$ and boundedness of $\|\e(\ul v)\|_{L^2(0,\ol t_\tau;L^2)}^2$ by the
				First a priori estimates,
				$$
					T_1^{(2,2)}\geq-\delta\int_0^{\ol t_\tau}\|\partial_t\chi\|_{H^1}^2\|\e(\ul v)\|_{L^2}^2\ds-\delta\|\e(\ul v)\|_{L^2(0,\ol t_\tau;H^1)}^2-C_\delta.
				$$
				Please notice that in order to treat the term $T_1^{(2,2)}$ in the sequel
				it will be crucial to have established the boundedness of $\|\partial_t\chi\|_{L^2(0,T;H^1)}$
				(see First a priori estimates) which is due to the higher-order viscosity term $-\Delta\chi_t$ in the damage equation.
				
				The term $T_1^{(3)}$ can be treated by using the Neumann condition \eqref{eqn:disc4} and by applying the discrete integration by parts formula in time
				\begin{align*}
					T_1^{(3)}
						&=-\int_{0}^{\ol t_\tau}\int_\Gamma\partial_t v\cdot\ol b\dxs\\
						&=-\int_{0}^{\ol t_\tau}\int_\Gamma\ul v\cdot\partial_t b\dxs
							-\int_\Gamma\ol v(t)\cdot \ol b(t)\dx
							+\int_\Gamma v^0\cdot b^0\dx\\
						&\geq -\int_{0}^{\ol t_\tau}\|\ul v\|_{L^2(\Gamma;\Rn)}\|\partial_t b\|_{L^2(\Gamma;\Rn)}\ds
							-\|\ol v(t)\|_{L^2(\Gamma;\Rn)}\|\ol b(t)\|_{L^2(\Gamma;\Rn)}\\
						&\quad-\|v^0\|_{L^2(\Gamma;\Rn)}\|b^0\|_{L^2(\Gamma;\Rn)}
				\end{align*}
 				By using the trace theorem $H^{1}(\Omega;\R^n)\hookrightarrow L^2(\Gamma;\R^n)$
				and the boundedness of\linebreak
				$\|\ul v\|_{L^2(0,\ol t_\tau;H^1)}^2$
				(see First a priori estimate) as well as of $\|\partial_t b\|_{L^2(0,\ol t_\tau;L^{2}(\Gamma;\Rn))}^2$, 
				$\|\ol b(t)\|_{L^2(\Gamma;\Rn)}^2$, $\|v^0\|_{H^1}^2$ and $\|b^0\|_{L^2(\Gamma;\Rn)}^2$,
				we obtain
				\begin{align*}
					T_1^{(3)}\geq{}&-\frac12\|\ul v\|_{L^2(0,\ol t_\tau;H^1)}^2-\frac12\|\partial_t b\|_{L^2(0,\ol t_\tau;L^{2}(\Gamma;\Rn))}^2
						-\delta\|\ol v(t)\|_{H^1}^2-C_\delta\|\ol b(t)\|_{L^2(\Gamma;\Rn)}^2\notag\\
						&-\frac12\|v^0\|_{H^1}^2-\frac12\|b^0\|_{L^2(\Gamma;\Rn)}^2\\
						\geq{}&-C_\delta-\delta\|\ol v(t)\|_{H^1}^2.
				\end{align*}
				\item[]\hspace*{-1.5em}-- To $(T_2)$:
					With the help of Young's inequality, we estimate $T_2$ by
					\begin{align}
						T_2\geq{}&\delta\underbrace{\int_0^{\ol t_\tau}\int_\Omega\Big|\mathsf{d}(\ol\chi)\di\big(\DD\e(\ol v)\big)\Big|^2\dxs}_{=:T_2^{(1)}}
							-C_\delta\underbrace{\int_0^{\ol t_\tau}\int_\Omega\Big|\mathsf{c}'(\ol\chi)\nabla \ol\chi\cdot\CC\e(\ol u)\Big|^2\dxs}_{=:T_2^{(2)}}\notag\\
							&-C_\delta\bigg(\underbrace{\int_0^{\ol t_\tau}\int_\Omega\Big|\mathsf{c}(\ol\chi)\di\big(\CC\e(\ol u)\big)\Big|^2\dxs}_{=:T_2^{(3)}}
			 				+\underbrace{\int_0^{\ol t_\tau}\int_\Omega\Big|\mathsf{d}'(\ol\chi)\nabla \ol\chi\cdot\DD\e(\ol v)\Big|^2\dxs}_{=:T_2^{(4)}}
							\bigg)
						\label{eqn:T2term}
					\end{align}
					By the following elliptic regularity estimate
					which follows from \cite[Theorem 3.4.1]{CDN10}
					(remember that $\Gamma$ is a $C^2$-boundary by (A1))
					$$
						\|w\|_{H^2}^2\leq C\Big(\|\di(\DD\e(w))\|_{L^2}^2+\|w\|_{H^1}^2+\|\DD\e(w)\cdot \nu \|_{H^{1/2}(\Gamma;\Rn)}^2\Big),
					$$
					valid for all $w\in H^2(\Omega;\Rn)$, by the Neumann boundary condition \eqref{eqn:disc4} (remember that $\mathsf{d}(\cdot)\geq\eta>0$ by (A3))
					and by the boundedness of $\|\ol v\|_{L^2(0,T;H^1)}$ (see First a priori estimate),
					we obtain
					\begin{align*}
						T_2^{(1)}\geq{}&\widetilde C\int_0^{\ol t_\tau}\int_\Omega\big|\di(\DD\e(\ol v))\big|^2\dxs\\
						\geq{}& \widetilde C\|\ol v\|_{L^2(0,\ol t_\tau;H^2)}^2-C\|\ol v\|_{L^2(0,\ol t_\tau;H^1)}^2-C\|\DD\e(\ol v)\cdot\nu \|_{L^2(0,\ol t_\tau;H^{1/2}(\Gamma;\Rn))}^2\\
						={}&\widetilde C\|\ol v\|_{L^2(0,\ol t_\tau;H^2)}^2
							-C\|\ol v\|_{L^2(0,\ol t_\tau;H^1)}^2
							-C\big\|\frac{\ol b}{\mathsf{d}(\ol\chi)}-\frac{\mathsf{c}(\ol\chi)}{\mathsf{d}(\ol\chi)}\CC\e(\ol u)\cdot\nu\big\|_{L^2(0,\ol t_\tau;H^{1/2}(\Gamma;\Rn))}^2\\
						={}&\widetilde C\|\ol v\|_{L^2(0,\ol t_\tau;H^2)}^2\\
							&-C\Big(\underbrace{\big\|\frac{\ol b}{\mathsf{d}(\ol\chi)}\big\|_{L^2(0,\ol t_\tau;H^{1/2}(\Gamma;\Rn))}^2}_{T_1^{(1,1)}}
								+\underbrace{\|\frac{\mathsf{c}(\ol\chi)}{\mathsf{d}(\ol\chi)}\CC\e(\ol u)\cdot\nu\big\|_{L^2(0,\ol t_\tau;H^{1/2}(\Gamma;\Rn))}^2}_{T_1^{(1,2)}}
								+1\Big).
					\end{align*}
					The constant $\widetilde C>0$ does depend on $\eta$ from (A3).
					The well-known trace theorem yields $H^{1}(\Omega;\Rn)\hookrightarrow H^{1/2}(\Gamma;\Rn)$
					with a continuous right inverse $H^{1/2}(\Gamma;\Rn)\hookrightarrow H^{1}(\Omega;\Rn)$
					(see \cite[Theorem 8.8]{Wl87}).
					In the following, we denote the extension of $\ol b$ also by $\ol b$.
					We obtain by using the trace theorem, the Gagliardo-Nirenberg type inequality \eqref{eqn:GN},
					the assumptions in (A2) and (A3)
					and the boundedness of $\|\nabla\ol \chi\|_{L^\infty(0,T;L^2)}$
					and $\|\e(\ol u)\|_{L^\infty(0,T;L^2)}$ (see First a priori estimates)
					\begin{align*}
						T_2^{(1,1)}
							\leq{}&
							C\big\|\frac{\ol b}{\mathsf{d}(\ol\chi)}\big\|_{L^2(0,\ol t_\tau;H^{1})}^2\\
							\leq{}&
							C\Big(\big\|\ol b\big\|_{L^2(0,\ol t_\tau;L^{2})}^2
								+\big\|\nabla \ol b\big\|_{L^2(0,\ol t_\tau;L^{2})}^2
								+\big\||\ol b||\nabla\ol\chi|\big\|_{L^2(0,\ol t_\tau;L^{2})}^2
							\Big)\\
							\leq{}&
							C\Big(\big\|\ol b\big\|_{L^2(0,\ol t_\tau;H^{1})}^2
								+\int_0^{\ol t_\tau}\|\ol b\|_{L^4}^2\|\nabla\ol\chi\|_{L^4}^2\ds
							\Big)\\
							\leq{}&
							C\Big(\big\|\ol b\big\|_{L^2(0,\ol t_\tau;H^{1})}^2
								+\int_0^{\ol t_\tau}\|\ol b\|_{L^2}\|\ol b\|_{H^1}\|\nabla\ol\chi\|_{L^2}\|\nabla\ol\chi\|_{H^1}\ds
							\Big)\\
							\leq{}&
							C\big\|\ol b\big\|_{L^2(0,\ol t_\tau;H^{1})}^2
								+\delta\int_0^{\ol t_\tau}\|\ol b\|_{L^2}^2\|\nabla\ol\chi\|_{H^1}^2\ds
								+C_\delta\|\ol b\|_{L^2(0,\ol t_\tau;H^1)}^2\\
							\leq{}&
							C_\delta\big\|\ol b\big\|_{L^2(0,\ol t_\tau;H^{1/2}(\Gamma;\Rn))}^2
								+\delta\int_0^{\ol t_\tau}\|\ol b\|_{H^{1/2}(\Gamma;\Rn)}^2\|\nabla\ol\chi\|_{H^1}^2\ds.
					\end{align*}
					as well as
					\begin{align*}
						T_2^{(1,2)}\leq{}&
							C\big\|\frac{\mathsf{c}(\ol\chi)}{\mathsf{d}(\ol\chi)}\CC\e(\ol u)\cdot\nu\big\|_{L^2(0,\ol t_\tau;H^{1})}^2\\
							\leq{}&
							C\Big(\|\e(\ol u)\|_{L^2(0,\ol t_\tau;L^2)}^2
								+\||\e(\ol u)||\nabla\ol\chi|\|_{L^2(0,\ol t_\tau;L^2)}^2
								+\|\nabla(\CC\e(\ol u))\|_{L^2(0,\ol t_\tau;L^2)}^2
							\Big)\\
							\leq{}&
								C\Big(1+\int_0^{\ol t_\tau}\|\e(\ol u)\|_{L^4}^2\|\nabla\ol\chi\|_{L^4}^2\ds
								+\|\ol u\|_{L^2(0,\ol t_\tau;H^2)}^2\Big)\\
							\leq{}&
								C\Big(1+\int_0^{\ol t_\tau}\|\e(\ol u)\|_{L^2}\|\e(\ol u)\|_{H^1}\|\nabla\ol\chi\|_{L^2(\Omega)}\|\nabla\ol\chi\|_{H^1}\ds
								+\|\ol u\|_{L^2(0,\ol t_\tau;H^2)}^2\Big)\\
							\leq{}&
								C\Big(1+\int_0^{\ol t_\tau}\big(\|\e(\ol u)\|_{H^1}^2+\|\nabla\ol\chi\|_{H^1}^2\big)\ds
								+\|\ol u\|_{L^2(0,\ol t_\tau;H^2)}^2\Big)\\
							\leq{}&
								C\Big(1+\int_0^{\ol t_\tau}\big(\|\ol u\|_{H^2}^2+\|\nabla\ol\chi\|_{H^1}^2\big)\ds\Big).
					\end{align*}
				We estimate the remaining terms in \eqref{eqn:T2term} by using again the Gagliardo-Nirenberg type inequality \eqref{eqn:GN}, the boundedness
				of $\mathsf{c}$, $\mathsf{c}'$ and $\mathsf{d}$ (see (A2) and (A3)),
				and the boundedness of $\|\nabla\ol \chi\|_{L^\infty(0,T;L^2)}$
				and $\|\e(\ol u)\|_{L^\infty(0,T;L^2)}$ (see First a priori estimates)
				\begin{align*}
					T_2^{(2)}
						\leq{}& C\int_0^{\ol t_\tau}\big(\|\e(\ol u)\|_{H^1}^2+\|\nabla\ol\chi\|_{H^1}^2\big)\ds,\\
					T_2^{(3)}\leq{}& C\int_0^{\ol t_\tau}\|\ol u\|_{H^2}^2\ds,\\
					T_2^{(4)}\leq{}& C\int_0^{\ol t_\tau}\|\nabla \ol\chi\|_{L^4}^2\|\e(\ol v)\|_{L^4}^2\ds\\
						\leq{}& C\int_0^{\ol t_\tau}\|\nabla \ol\chi\|_{L^2}\|\nabla \ol\chi\|_{H^1}\|\e(\ol v)\|_{L^2}\|\e(\ol v)\|_{H^1}\ds\\
						\leq{}& C_\delta\int_0^{\ol t_\tau}\|\e(\ol v)\|_{L^2}^2\|\nabla \ol\chi\|_{H^1}^2\ds
							+\delta\|\e(\ol v)\|_{L^2(0,\ol t_\tau;H^1)}^2.
				\end{align*}

				\item[]\hspace*{-1.5em}-- To $(T_3)$:
				It can be seen by integration by parts and from the definition of $I_\beta$ (see Definition \ref{def:regularization} (iii)-(iv)) that
				\begin{align*}
					T_3=\int_0^{\ol t_\tau}\int_\Omega\xi_\beta'(\partial_t\chi)|\nabla\partial_t\chi|^2\dxs\geq 0.
				\end{align*}
				\item[]\hspace*{-1.5em}-- To $(T_4)$:
				The term $T_4$ can be treated by applying the Gagliardo-Nirenberg type inequality \eqref{eqn:GN},
				by using the boundedness of $\mathsf{c}_1'$ and of $\mathsf{c}_2'$ (see assumption (A2))
				and by using boundedness of $\|\e(\ul u)\|_{L^\infty(0,T;L^2)}$ (see First a priori estimate).
				We obtain
				\begin{align*}
					T_4\geq{}&-\delta\|\Delta\partial_t\chi\|_{L^2(0,\ol t_\tau;L^2)}^2-C_\delta\int_0^{\ol t_\tau}\|\e(\ul u)\|_{L^4}^4\ds\\
					\geq{}&-\delta\|\Delta\partial_t\chi\|_{L^2(0,\ol t_\tau;L^2)}^2-C_\delta\int_0^{\ol t_\tau}\|\e(\ul u)\|_{L^2}^2\|\e(\ul u)\|_{H^1}^2\ds\\
					\geq{}&-\delta\|\Delta\partial_t\chi\|_{L^2(0,\ol t_\tau;L^2)}^2-C_\delta\|\e(\ul u)\|_{L^2(0,\ol t_\tau;H^1)}^2.
				\end{align*}
				\item[]\hspace*{-1.5em}-- To $(T_5)$:
				We find by Young's inequality, by the Lipschitz continuity of $f'$ (see (A4)) and by the
				boundedness of $\|\ol\chi\|_{L^2(0,T;L^2)}$ (see First a priori estimates):
				\begin{align*}
					T_5\geq{}&-\delta\|\Delta\partial_t\chi\|_{L^2(0,\ol t_\tau;L^2)}^2-C_\delta(\|\ol\chi\|_{L^2(0,T;L^2)}^2+1)\\
						\geq{}&-\delta\|\Delta\partial_t\chi\|_{L^2(0,\ol t_\tau;L^2)}^2-C_\delta.
				\end{align*}
			\end{itemize}
			In the following we use the estimates (by the fundamental theorem of calculus and H\"older's inequality)
			\begin{subequations}
			\label{eqn:FTEst}
			\begin{align}
				&\|\ol u(s)\|_{H^2}^2=\Big\|u^0+\int_0^{\ol s_\tau}\partial_t u(\iota)\mathrm d\iota\Big\|_{H^2}^2
					\leq\|u^0\|_{H^2}^2+\ol s_\tau\|\ol v\|_{L^2(0,\ol s_\tau;H^2)}^2,\\
				&\|\ol\chi(s)\|_{H^2}^2
					\leq\|\chi^0\|_{H^2}^2+\ol s_\tau\|\partial_t\chi\|_{L^2(0,\ol s_\tau;H^2)}^2.
			\end{align}
			\end{subequations}			
			Now we conclude by taking the above estimates into account:
			\begin{align*}
				T_1\geq{}&\widetilde\eta\|\ol v(t)\|_{H^1}^2
					-D_\delta
					-\delta\big(\|u\|_{H^1(0,\ol t_\tau;H^2)}^2+\|\ol v(t)\|_{H^1}^2\big)
					-\delta\int_0^{\ol t_\tau}\|\partial_t\chi(s+\tau)\|_{H^1}^2\|\ol v(s)\|_{H^1}^2\ds,\\
				T_2\geq{}&\widetilde\eta\|u\|_{H^1(0,\ol t_\tau;H^2)}^2
					-D_\delta
					-C_{\delta}\int_0^{\ol t_\tau}\Big(\|\ol\chi\|_{H^2}^2+\|u\|_{H^1(0,\ol s_\tau;H^2)}^2
						+\ol t_\tau\|\ol v\|_{H^1}^2\|\partial_t\chi\|_{L^2(0,\ol s_\tau;H^2)}^2\Big)\ds\\
					&-\delta\|u\|_{H^1(0,\ol t_\tau;H^2)}^2
						-\delta\int_0^{\ol t_\tau}\|\ol b\|_{H^{1/2}(\Gamma;\Rn)}^2\|\ol\chi\|_{H^2}^2\ds,\\
				T_3\geq{}&0,\\
				T_4\geq{}&-D_\delta-C_\delta\int_0^{\ol t_\tau}\|u\|_{H^1(0,\ol s_\tau;H^2)}^2\ds-\delta\|\chi\|_{H^1(0,\ol t_\tau;H^2)}^2,\\
				T_5\geq{}&-C_{\delta}-\delta\|\chi\|_{H^1(0,\ol t_\tau;H^2)}^2,
			\end{align*}
			where the constant $C_\delta>0$ continuously depends on (besides $\delta$)
			$$
				C_\delta=C_\delta\big(\|u^0\|_{H^1},\|v^0\|_{L^2},\|\chi^0\|_{H^1}, \|b\|_{L^{2}(0,T;L^{2}(\Gamma;\Rn))},\|\ell\|_{L^2(0,T;L^2)}\big)
			$$
			and the constant $D_\delta>0$ continuously depends on (besides $\delta$)
			$$
				D_\delta=D_\delta\big(\|u^0\|_{H^2},\|v^0\|_{H^1},\|\chi^0\|_{H^2}, \|b\|_{L^{2}(0,T;H^{1/2}(\Gamma;\Rn))\cap H^{1}(0,T;L^{2}(\Gamma;\Rn))},\|\ell\|_{L^2(0,T;L^2)}\big).
			$$

			By adding the identities \eqref{eqn:est6} and \eqref{eqn:est4}, using the estimates for $T_1,\ldots,T_5$ developed above,
			using the $H^2$-regularity estimate (see \cite[Theorem 3.4.1]{CDN10})
			$$
				\|w\|_{H^2}^2\leq C\big(\|\Delta w\|_{L^2}^2+\|w\|_{H^1}^2\big)\quad\text{valid for all }w\in H_{N}^2(\Omega)
			$$
			applied to $\ol\chi$ and $\partial_t\chi$ (note the boundary conditions in \eqref{eqn:disc5a}),
			we obtain
			\begin{align}
				&\|\ol v(t)\|_{H^1}^2
					+\|\ol\chi(t)\|_{H^2}^2
					+\|u\|_{H^1(0,\ol t_\tau;H^2)}^2
					+\|\chi\|_{H^1(0,\ol t_\tau;H^2)}^2\notag\\
					&+\|\di\big(\mathsf{c}(\ol\chi)\CC\e(\ol u)+\mathsf{d}(\ol\chi)\DD\e(\ol v)\big)\|_{L^2(0,\ol t_\tau;L^2)}^2\notag\\
				&\leq D_\delta
						+C_{\delta}\int_0^{\ol t_\tau}\Big(\|\ol\chi\|_{H^2}^2+\|u\|_{H^1(0,\ol s_\tau;H^2)}^2+\ol t_\tau\|\ol v\|_{H^1}^2\|\partial_t\chi\|_{L^2(0,\ol s_\tau;H^2)}^2\Big)\ds\notag\\
					&\quad
						+\delta\Big(\|u\|_{H^1(0,\ol t_\tau;H^2)}^2+\|\ol v(t)\|_{H^1}^2
						+\|\di\big(\mathsf{c}(\ol\chi)\CC\e(\ol u)+\mathsf{d}(\ol\chi)\DD\e(\ol v)\big)\big)\|_{L^2(0,\ol t_\tau;L^2)}^2
						+\|\chi\|_{H^1(0,\ol t_\tau;H^2)}^2\Big)\notag\\
					&\quad
						+\delta\int_0^{\ol t_\tau}\|\partial_t\chi(s+\tau)\|_{H^1}^2\|\ol v(s)\|_{H^1}^2\ds
						+\delta\int_0^{\ol t_\tau}\|\ol b\|_{H^{1/2}(\Gamma;\Rn)}^2\|\ol\chi\|_{H^2}^2\ds.
			\label{eqn:finalEst}
			\end{align}

			By choosing $\delta>0$ small, the first $\delta$-term on the right-hand side of
			\eqref{eqn:finalEst} can be absorbed by the left-hand side.
			Furthermore, for later estimates, $\delta$ should also satisfy
			\begin{align}
				\delta<\frac{1}{8\big(\|\partial_t\chi\|_{L^2(0,T;H^1)}^2+\|\ol b\|_{L^2(0,T;H^{1/2}(\Gamma;\Rn))}^2+1\big)}.
			\label{eqn:deltaEst}
			\end{align}
			Indeed, the denominator of the right-hand side is bounded from above
			by the First a priori estimates, hence the right-hand side is bounded from below
			and $\delta>0$ can be chosen such that \eqref{eqn:deltaEst} holds.

			We infer from the estimates \eqref{eqn:finalEst} and \eqref{eqn:deltaEst}
			\begin{align}
			\label{eqn:finalEst2}
				\alpha^k\leq D_\delta+\sum_{j=1}^k\tau\gamma^j\alpha^j
			\end{align}
			with
			\begin{align*}
				\alpha^k:={}&\|v^k\|_{H^1}^2
					+\|\chi^k\|_{H^2}^2
					+\|u\|_{H^1(0,\tau k;H^2)}^2
					+\|\chi\|_{H^1(0,\tau k;H^2)}^2,\\
				\gamma^k:={}&C_\delta
					+\frac{\|(\chi^{k+1}-\chi^{k})/\tau\|_{H^1}^2+\|b^k\|_{H^{1/2}(\Gamma;\Rn)}^2}{8(\|\partial_t\chi\|_{L^2(0,T;H^1)}^2
					+\|\ol b\|_{L^2(0,T;H^{1/2}(\Gamma;\Rn))}^2+1)}
					+C_\delta \tau k\|v^k\|_{H^1}^2.
			\end{align*}
			
			
			In the following, we will choose a time $t_0>0$ such that
			for all small $\tau>0$ and all $k=1,\ldots,{(\ol t_0)}_\tau/\tau$:
			\begin{align}
				0\leq\tau\gamma^k<\frac 12.
			\label{eqn:gammaEst}
			\end{align}
			Indeed,
			we know by the First a priori estimate that
			\begin{align*}
				&\sum_{k=1}^{M}\tau\|v^k\|_{H^1}^2<\widehat C
				\qquad\text{uniformly in $\tau$,}
			\end{align*}
			where $\widehat C>0$ denotes the constant $C$ in \eqref{eqn:aprioriC}.
			Thus
			\begin{align}
				\tau \|v^k\|_{H^1}^2<\widehat C
				\qquad\text{uniformly in $\tau$ and in $k$.}
			\label{eqn:omegaEst}
			\end{align}
			By choosing
			\begin{align}
			\label{eqn:tZeroConst}
				t_0:=\frac{1}{4C_\delta \widehat C},
			\end{align}
			we get for all $k=1,\ldots,{(\ol t_0)}_\tau/\tau$:
			\begin{align*}
				\tau\gamma^k
				={}&\tau C_\delta+\frac{\tau \|(\chi^{k+1}-\chi^{k})/\tau\|_{H^1}^2+\tau\|b^k\|_{H^{1/2}(\Gamma;\Rn)}^2}{8\big(\sum_{j=0}^{M-1}\big(\tau\|(\chi^{j+1}-\chi^{j})/\tau\|_{H^1}^2+\tau\|b^{j+1}\|_{H^{1/2}(\Gamma;\Rn)}^2\big)+1\big)}\\
					&+\tau C_\delta\;\times\underbrace{\tau k}_{\leq{(\ol t_0)}_\tau\leq t_0+\tau}\times\;\|v^k\|_{H^1}^2\\
				\leq{}& \tau C_\delta+\frac18+\tau C_\delta(t_0+\tau)\|v^k\|_{H^1}^2\\
				\leq{}&\tau C_\delta+\frac18+\frac{1}{4}+\tau C_\delta\widehat C\\
				\leq{}& \tau C_\delta(1+\widehat C)+\frac18+\frac{1}{4}.
			\end{align*}
			Consequently, for small $\tau>0$, estimate \eqref{eqn:gammaEst} is fulfilled.
			
			Finally, by ensuring \eqref{eqn:gammaEst}, \eqref{eqn:finalEst2} rewrites in the desired form
			\begin{align*}
				\alpha^k\leq \frac{D_\delta}{1-\tau\gamma^k}+\sum_{j=1}^{k-1}\tau\frac{\gamma^j}{1-\tau\gamma^k}\alpha^j
			\end{align*}
			and, therefore,
			\begin{align*}
				\alpha^k\leq \frac{D_\delta}{2}+\sum_{j=1}^{k-1}\tau\frac{\gamma^j}{2}\alpha^j.
			\end{align*}
			We are now in a position to apply the discrete version of Gronwall's lemma in the sum form (see, e.g., \cite[page 26]{Rou13}) and obtain
			\begin{align*}
				\alpha^k\leq\frac{D_\delta}{2}e^{\sum_{j=1}^{k-1}\tau\frac{\gamma^j}{2}}.
			\end{align*}
			We obtain boundedness of $\alpha^k$ uniformly in $\tau$ and $k=1,\ldots, (\ol t_0)_\tau/\tau$.
			Therefore, (ii) is shown except the boundedness for $\|\xi_\beta(\partial_t\chi_{\tau,\beta})\|_{L^2(0,(\ol t_0)_\tau;L^2)}$.
			The latter follows by a comparison argument in \eqref{eqn:disc2}.

		\item[]\hspace*{-1em}\underline{To (ii) -- global-in-time estimate:}\\
		
		The main observation to obtain global-in-time estimates is that
		the local estimates above can not only be performed on the time interval
		$[0,(\ol t_0)_\tau]$ but also, with minor modifications, to each interval $[\ul s_\tau, \ol t_\tau]\subseteq [0,T]$
		such that $|\ol t_\tau-\ul s_\tau|\leq t_0$, where $t_0>0$ from \eqref{eqn:tZeroConst} depends on
		quantities which can be bounded globally in time by the First a priori estimates.
		Thus we find a $t_0>0$ such that the Second a priori estimates can be performed
		on each interval interval $[\ul s_\tau, \ol t_\tau]\subseteq[0,T]$ with $|\ol t_\tau-\ul s_\tau|\leq t_0$.
		
		To conclude the proof, let
		\begin{align*}
			&\frak t_\tau^k:=\left(\ul{k\frac{t_0}{2}}\right)_\tau
				=\max\Big\{j\tau\,\big|\,j\in\N\text{ such that }k\frac{t_0}{2}\geq\tau j\Big\},\\
			&\frak l_\tau:=\left(\ol{t_0}\right)_\tau
				=\min\Big\{j\tau\,\big|\,j\in\N\text{ such that }t_0\leq\tau j\Big\}.
		\end{align*}
		We
		define the time intervals
		\begin{align*}
			I_\tau^k:=[\frak t_\tau^k,\frak t_\tau^k+\frak l_\tau]\cap[0,T]
		\end{align*}
		for all $k=0,\ldots,N$ with $N:=\lceil T/(t_0/2)\rceil-1$ where
		$\lceil\cdot\rceil$ denotes the ceiling function.
		
		
		We apply the local-in-time estimates above to each interval $I_\tau^k$ and
		obtain constants $C_0,\ldots, C_N>0$ which continuously depend on
		\begin{align*}
			C_k=C_k\big(&\|\ol u_{\tau,\beta}(\frak t_\tau^k)\|_{H^2},\|\ol v_{\tau,\beta}(\frak t_\tau^k)\|_{H^1},
			\|\ol \chi_{\tau,\beta}(\frak t_\tau^k)\|_{H^2},\\
				&\|b\|_{L^{2}(0,T;H^{1/2}(\Gamma;\Rn))\cap H^{1}(0,T;L^{2}(\Gamma;\Rn))}, \|\ell\|_{L^2(0,T;L^2)}\big),\qquad k=1,\ldots,N
		\end{align*}
		such that for all $\tau,\beta>0$ and all $k=1,\ldots,N$
		\begin{align*}
			&		\|u_{\tau,\beta}\|_{H^1(I_\tau^k;H^{2})\cap W^{1,\infty}(I_\tau^k;H^1)}\leq C_k,
			&&	\|\chi_{\tau,\beta}\|_{H^1(I_\tau^k;H^{2})}\leq C_k,\\
			&		\|\ul{u}_{\tau,\beta}\|_{L^\infty(I_\tau^k;H^{2})}\leq C_k,
			&&	\|\ul \chi_{\tau,\beta}\|_{L^\infty(I_\tau^k;H^{2})}\leq C_k,\\
			&		\|\ol{u}_{\tau,\beta}\|_{L^\infty(I_\tau^k;H^{2})}\leq C_k,
			&&	\|\ol \chi_{\tau,\beta}\|_{L^\infty(I_\tau^k;H^{2})}\leq C_k,\\
			&		\|v_{\tau,\beta}\|_{L^2(I_\tau^k;H^2)\cap L^{\infty}(I_\tau^k;H^1)\cap H^1(I_\tau^k;L^{2})}\leq C_k,
			&&	\|\xi_\beta(\partial_t\chi_{\tau,\beta})\|_{L^2(I_\tau^k;L^2)}\leq C_k.
		\end{align*}
		
		To obtain a global bound, we can argue by induction. We sketch the argument:
		
		Suppose we have given the a priori bound $C_{k-1}$ for the time interval $I_\tau^{k-1}$.
		By definition, we find $\frak t_\tau^k\in I_\tau^{k-1}$.
		Thus
		\begin{align*}
			&\|\ol u_{\tau,\beta}(\frak t_\tau^k)\|_{H^2} \leq \|\ol{u}_{\tau,\beta}\|_{L^\infty(I_\tau^{k-1};H^{2})}\leq C_{k-1},\\
			&\|\ol v_{\tau,\beta}(\frak t_\tau^k)\|_{H^1} \leq \|\ol v_{\tau,\beta}\|_{L^{\infty}(I_\tau^{k-1};H^1)}\leq C_{k-1},\\
			&\|\ol \chi_{\tau,\beta}(\frak t_\tau^k)\|_{H^2}\leq \|\ol \chi_{\tau,\beta}\|_{L^\infty(I_\tau^{k-1};H^{2})}\leq C_{k-1}.
		\end{align*}
		Consequently, we find an a priori bound $\widetilde C_k\geq C_k$ for the solutions on the interval $I_\tau^k$ by
		$$
			\widetilde C_k:=\max_{|x|,|y|,|z|\leq C_{k-1}} C_{k}(x,y,z, \|b\|_{L^{2}(0,T;H^{1/2}(\Gamma;\Rn))\cap H^{1}(0,T;L^{2}(\Gamma;\Rn))},\|\ell\|_{L^2(0,T;L^2)}).
		$$
		Note that $\widetilde C_k$ does only depend on
		\begin{align*}
			\widetilde C_k=\widetilde C_k\big(&\|\ol u_{\tau,\beta}(\frak t_\tau^{k-1})\|_{H^2},\|\ol v_{\tau,\beta}(\frak t_\tau^{k-1})\|_{H^1},
				\|\ol \chi_{\tau,\beta}(\frak t_\tau^{k-1})\|_{H^2},\\
 				&\|b\|_{L^{2}(0,T;H^{1/2}(\Gamma;\Rn))\cap H^{1}(0,T;L^{2}(\Gamma;\Rn))},\|\ell\|_{L^2(0,T;L^2)}\big).
		\end{align*}

		\ep
	\end{enumerate}
	
	We perform the limit passage $\tau\downarrow 0$ and $\beta\downarrow 0$ separately in order to
	show existence of strong solutions for both cases: namely for $\beta>0$ and $\beta=0$ in Definition \ref{def:notionSolution} (i) and (i).
	The a~priori estimates give rise to the following convergence properties along a suitably chosen subsequence.
	\begin{lemma}[Convergence properties]
	\label{lemma:conv}\hspace*{0.01em}\\
		There exist limit functions for every $\beta\geq 0$ (we will also write $u:=u_0$, $\chi:=\chi_0$)
		\begin{align*}
			&u_\beta\in H^1(0,T;H^2(\Omega;\R^n))\cap W^{1,\infty}(0,T;H^1(\Omega;\R^n))\cap H^2(0,T;L^2(\Omega;\R^n)),\\
			&\chi_\beta\in H^1(0,T;H^2(\Omega))
		\end{align*}
		with
		\begin{align*}
			&u_\beta(0)=u^0\text{ a.e. in }\Omega,&&\partial_t u_\beta(0)=v^0\text{ a.e. in }\Omega,
			&&\chi_\beta(0)=\chi^0\text{ a.e. in }\Omega,\\
			&u_\beta=b\text{ a.e. on }\Sigma
		\end{align*}
		such that
		\begin{itemize}
			\item[(i)]
			for fixed $\beta>0$ and $\tau\downarrow 0$ (along a subsequence):
			\begin{subequations}
			\begin{align}
			\label{eqn:uConv1}
				&u_{\tau,\beta}\to u_\beta&&\textit{ weakly in }H^1(0,T;H^2(\Omega;\R^n))\\
					&&&\textit{ weakly-star in } W^{1,\infty}(0,T;H^1(\Omega;\R^n)),\\
			\label{eqn:uConv1b}
				&\overline{u}_{\tau,\beta},\underline{u}_{\tau,\beta}\to u_\beta&&\textit{ weakly-star in }L^\infty(0,T;H^2(\Omega;\R^n)),\\
			\label{eqn:uConv2}
				&u_{\tau,\beta}\to u_\beta&&\textit{ strongly in }H^1(0,T;H^{1}(\Omega;\R^n)),\\
			\label{eqn:uConv3}
				&\overline{u}_{\tau,\beta},\underline{u}_{\tau,\beta}\to u_\beta&&\textit{ strongly in }L^\infty(0,T;H^{1}(\Omega;\R^n)),\\
			\label{eqn:uConv4}
				&u_{\tau,\beta}, \overline{u}_{\tau,\beta},\underline{u}_{\tau,\beta}\to u_\beta&&\textit{ a.e. in }\Omega\times(0,T),\\
				&v_{\tau,_\beta}\to \partial_t u_\beta&&\textit{ weakly in }H^1(0,T;L^2(\Omega;\R^n)),\\
				&\chi_{\tau,\beta}\to\chi_\beta&&\textit{ weakly in } H^1(0,T;H^2(\Omega)),\\
					&\overline{\chi}_{\tau,\beta},\underline{\chi}_{\tau,\beta}\to\chi_\beta&&\textit{ weakly-star in }L^\infty(0,T;H^{2}(\Omega)),\\
			\label{eqn:chiConv4a}
					&\overline{\chi}_{\tau,\beta},\underline{\chi}_{\tau,\beta}\to\chi_\beta&&\textit{ strongly in }L^\mu(0,T;H^{1}(\Omega))\text{ for all }\mu\geq 1,\\
			\label{eqn:chiConv4}
					&\overline{\chi}_{\tau,\beta},\underline{\chi}_{\tau,\beta}\to\chi_\beta&&\textit{ uniformly on }\ol\Omega\times[0,T],\\
			\label{eqn:xiConv1}
				&\xi_\beta(\partial_t\chi_{\tau,\beta})\to\xi_\beta(\partial_t\chi_\beta)&&\textit{ weakly in }L^2(0,T;L^{2}(\Omega)),
			\end{align}
			\end{subequations}
		\item[(ii)] for $\beta\downarrow 0$ (along a subsequence):
			\begin{subequations}
			\begin{align}
				&u_{\beta}\to u&&\textit{ weakly in }H^2(0,T;L^2(\Omega;\R^n))\cap H^1(0,T;H^2(\Omega;\R^n))\\
				&&&\textit{ weakly-star in }W^{1,\infty}(0,T;H^1(\Omega;\R^n)),\\
				&u_{\beta}\to u&&\textit{ strongly in }H^1(0,T;H^{1}(\Omega;\R^n)),\\
				&u_{\beta}\to u&&\textit{ a.e. in }\Omega\times(0,T),\\
			\label{eqn:chiWeakConv}
				&\chi_{\beta}\to\chi&&\textit{ weakly in }H^1(0,T;H^2(\Omega)),\\
				&\chi_{\beta}\to\chi&&\textit{ strongly in }L^\mu(0,T;H^{1}(\Omega))\text{ for all $\mu\geq 1$},\\
				&\chi_{\beta}\to\chi&&\textit{ uniformly on }\ol \Omega\times[0,T],\\
				&\xi_\beta(\partial_t\chi_{\beta})\to\xi&&\textit{ weakly in }L^2(0,T;L^{2}(\Omega))\text{ for a }\xi\in\L^2(0,T,L^2(\Omega))\notag\\
				&&&\;\,\text{with }\xi\in\partial I_{(-\infty,0]}(\partial_t\chi)\text{ a.e. in }\Omega\times(0,T).
				\label{eqn:xiConv2}
			\end{align}
			\end{subequations}
		\end{itemize}
	\end{lemma}
	\begin{proof}
		\begin{itemize}
		\item[]\hspace*{-1em}To (i):
		Properties \eqref{eqn:uConv1}-\eqref{eqn:uConv1b} and \eqref{eqn:uConv4}-\eqref{eqn:chiConv4a} can be obtained by standard compact embeddings,
		whereas \eqref{eqn:uConv2}, \eqref{eqn:uConv3} and \eqref{eqn:chiConv4} can be obtained by the Aubin-Lions type compactness result in \cite{Simon}
		(please note that $\ol v_{\tau,\beta}=\partial_t u_{\tau,\beta}$).
		It remains to show \eqref{eqn:xiConv1}.
		
		By Lemma \ref{lemma:aPrioriDiscr}, we find a cluster point $\eta_\beta\in L^2(0,T;L^{2}(\Omega))$ such that
		along a subsequence $\tau\downarrow 0$
		\begin{align}
			\xi_\beta(\partial_t\chi_{\tau,\beta})\to\eta_\beta\text{ weakly in }L^2(0,T;L^{2}(\Omega)).
		\label{eqn:etaWeakConv}
		\end{align}
		We have to show $\eta_\beta=\xi_\beta(\partial_t\chi_\beta)$ to finish the proof.
		This will be achieved by exploiting maximal monotonicity
		of the graph $\xi_\beta$ (also known as \textit{Minty's trick}).
		At first we observe that due to the monotonicity of $\xi_\beta$ viewed as an operator
		from $L^2(0,T;L^2(\Omega))$ to $L^2(0,T;L^2(\Omega))'$ one has
		\begin{align*}
			\big\langle \xi_\beta(\partial_t\chi_{\tau,\beta})-\xi_\beta(v),\partial_t\chi_{\tau,\beta}-v\big\rangle_{L^2(L^2)}\geq 0
		\end{align*}
		for all functions $v\in L^2(0,T;L^2(\Omega))$.
		In order to pass to the limit $\tau\downarrow 0$ it remains to show
		\begin{align}
			\limsup_{\tau\downarrow 0}\big\langle \xi_\beta(\partial_t\chi_{\tau,\beta}),\partial_t\chi_{\tau,\beta}\big\rangle_{L^2(L^2)}
			\leq\big\langle \eta_\beta,\partial_t\chi_\beta\big\rangle_{L^2(L^2)}.
		\label{eqn:limsupEst}
		\end{align}
		
		Testing equation \eqref{eqn:disc2} with $\partial_t\chi_{\tau,\beta}$ and integrating over $\Omega\times(0,T)$ in space and time
		and passing $\tau\downarrow 0$ by using weak lower-semicontinuity properties for the $\iint|\partial_t\chi_{\tau,\beta}|^2$-
		and $\iint|\nabla\partial_t\chi_{\tau,\beta}|^2$-term, we obtain
		\begin{align}
			&\liminf_{\tau\downarrow 0}\int_0^T\int_{\Omega}-\xi_\beta(\partial_t\chi_{\tau,\beta})\partial_t\chi_{\tau,\beta}\dxs\notag\\
			&\quad\geq\liminf_{\tau\downarrow 0}\int_0^T\int_{\Omega}|\partial_t\chi_{\tau,\beta}|^2\dxs
				+\liminf_{\tau\downarrow 0}\int_0^T\int_{\Omega}|\nabla\partial_t\chi_{\tau,\beta}|^2\dxs\notag\\
				&\qquad+\lim_{\tau\downarrow 0}\int_0^T\int_{\Omega}\nabla\ol\chi_{\tau,\beta}\cdot\nabla \partial_t\chi_{\tau,\beta}\dxs
				\notag\\
				&\quad\quad+\lim_{\tau\downarrow 0}
				\int_0^T\int_{\Omega}\left(\frac12\big(\mathsf{c}_1'(\ol\chi_{\tau,\beta})+\mathsf{c}_2'(\ul\chi_{\tau,\beta})\big)\CC\e(\ul u_{\tau,\beta}):\e(\ul u_{\tau,\beta})
				+f'(\ol\chi_{\tau,\beta})\right)\partial_t\chi_{\tau,\beta}\dxs\notag\\
			&\quad\geq\int_0^T\int_{\Omega}\Big(|\partial_t\chi_{\beta}|^2+|\nabla\partial_t\chi_{\beta}|^2\Big)\dxs
				+\int_0^T\int_{\Omega}\nabla\chi_{\beta}\cdot\nabla \partial_t\chi_{\beta}\dxs\notag\\
			&\qquad\quad
					+\int_0^T\int_{\Omega}\Big(\frac12\mathsf{c}'(\chi_{\beta})\CC\e(u_{\beta}):\e(u_{\beta})
					+f'(\chi_{\beta})\Big)\partial_t\chi_{\beta}\dxs\notag\\
			&\quad=\int_0^T\int_{\Omega}\left(\partial_t\chi_\beta
				-\Delta\chi_\beta-\Delta\partial_t\chi_\beta
				+\frac12\mathsf{c}'(\chi_{\beta})\CC\e(u_{\beta}):\e(u_{\beta})
				+f'(\chi_{\beta})\right)\partial_t\chi_{\beta}\dxs.
		\label{eqn:liminfEst2}
		\end{align}
		Note that we also get
		$$
			-\eta_\beta=\partial_t\chi_\beta
				-\Delta\chi_\beta-\Delta\partial_t\chi_\beta
				+\frac12\mathsf{c}'(\chi_{\beta})\CC\e(u_{\beta}):\e(u_{\beta})
				+f'(\chi_{\beta})\quad\text{a.e. in $\Omega\times(0,T)$}
		$$
		by performing a limit passage $\tau\downarrow 0$ in \eqref{eqn:disc2} after testing
		with a function, integrating and using the already
		known convergence properties \eqref{eqn:uConv1}-\eqref{eqn:chiConv4} and \eqref{eqn:etaWeakConv}.
		In combination with \eqref{eqn:liminfEst2} and multiplying by $-1$, we find the
		desired limsup-estimate \eqref{eqn:limsupEst}.
		Thus
		\begin{align}
			\big\langle\eta_\beta-\xi_\beta(v),\partial_t\chi_\beta-v\big\rangle_{L^2(L^2)}\geq 0
		\label{eqn:monotonicityBeta}
		\end{align}
		for all functions $v\in L^2(0,T;L^2(\Omega))$.
		By maximal monotonicity of $\xi_\beta=\partial I_\beta$ viewed as an operator $L^2(0,T;L^2(\Omega))\to L^2(0,T;L^2(\Omega))'$, we find
		$\eta_\beta=\xi_\beta(\partial_t\chi_\beta)$.

		\item[]\hspace*{-1em}To (ii):
		Since the a priori estimates in Lemma \ref{lemma:aPrioriDiscr} are also independent of $\beta$,
		we obtain an analogous result for the limit case $\beta\downarrow 0$.
	
		It remains to establish \eqref{eqn:xiConv2}.
		Convexity of $I_\beta$ viewed as a functional $L^2(0,T;L^2(\Omega))\to\R$ implies
		\begin{align}
			I_\beta(\partial_t\chi_\beta)+\big\langle\xi_\beta(\partial_t\chi_\beta),v-\partial_t\chi_\beta\big\rangle_{L^2(L^2)}\leq
			I_\beta(v)
		\label{eqn:convexityBeta}
		\end{align}
		for all functions $v\in L^2(0,T;L^2(\Omega))$.
		With the same arguments as in (i) we obtain the limsup-estimate
		\begin{align*}
			\limsup_{\beta\downarrow 0}\big\langle \xi_\beta(\partial_t\chi_\beta),\partial_t\chi_\beta\big\rangle_{L^2(L^2)}
			\leq\big\langle \eta,\partial_t\chi\big\rangle_{L^2(L^2)},
		\end{align*}
		where $\eta$ is a cluster point of $\{\partial_t\chi_\beta\}_{\beta\in(0,1)}$ in the weak topology of
		$L^2(0,T;L^2(\Omega))$.
		
		Now let $v\in L^2(0,T;L^2(\Omega))$ with $v\leq 0$ a.e. in $\Omega\times(0,T)$
		be arbitrary.
		By using $I_\beta(v)=0$ and $I_\beta(\partial_t\chi_\beta)\geq 0$,
		we obtain from \eqref{eqn:convexityBeta} after passing to $\beta\downarrow 0$
		for a subsequence
		\begin{align}
			\big\langle\eta,v-\partial_t\chi\big\rangle_{L^2(L^2)}\leq 0.
		\label{eqn:etaConv1}
		\end{align}
		We also infer from \eqref{eqn:convexityBeta} the following boundedness with respect to $\beta$:
		$$
			I_\beta(\partial_t\chi_\beta)\leq C.
		$$
		For the limit $\beta\downarrow 0$ we thus obtain
		\begin{align}
			\partial_t\chi\leq 0\text{ a.e. in }\Omega\times(0,T).
		\label{eqn:etaConv2}
		\end{align}
		Estimate \eqref{eqn:etaConv1} together with \eqref{eqn:etaConv2}
		yields $\eta\in\partial I_{L^2(0,T;L_-^2(\Omega))}(\partial_t\chi)$
		with the indicator function
		$I_{L^2(0,T;L_-^2(\Omega))}:L^2(0,T;L^2(\Omega))\to\R\cup\{+\infty\}$
		and $L_-^2(\Omega):=\{v\in L^2(\Omega)\,|\,v\leq 0\text{ a.e.}\}$.
		
		We end up with $\eta\in \partial I_{(-\infty,0]}(\partial_t\chi)$ a.e. in $\Omega\times(0,T)$.
		\ep
		\end{itemize}
	\end{proof}

	\begin{theorem}
	\label{theorem:existenceCont}
		Let the Assumptions (A1)-(A4) be satisfied and the data $(u^0,v^0,\chi^0,b,\ell)$ from \eqref{eqn:initialData} and \eqref{eqn:externalForces}
		be given.
		The following statements are true:
		\begin{enumerate}
			\item[(i)] Regularized case ($\beta>0$):
				There exists a strong global-in-time solution $(u_\beta,\chi_\beta)$ in the sense of Definition \ref{def:notionSolution} (ii)
				which satisfies \eqref{eqn:cont5a}.
			\item[(ii)] Limit case ($\beta=0$):
				There exists a strong global-in-time solution $(u,\chi)$ in the sense of Definition \ref{def:notionSolution} (i).
				which satisfies \eqref{eqn:cont5a}.
		\end{enumerate}
	\end{theorem}
	\begin{proof}
		\begin{enumerate}
			\item[]\hspace*{-1em}To (i):
				By multiplying the systems \eqref{eqn:disc1}, \eqref{eqn:disc2}, \eqref{eqn:disc4} and \eqref{eqn:disc5}
				with test-functions, integrating over space and time, we may pass to the limit $\tau\downarrow 0$
				for fixed $\beta>0$ by utilizing Lemma \ref{lemma:dataConv} and Lemma \ref{lemma:conv} (i)
				and standard convergence arguments.
				Then, switching back to an $(x,t)$-a.e. formulation, we obtain a strong solution of system
				\eqref{eqn:pdeSystem}-\eqref{eqn:IBC}.
				
				
			\item[]\hspace*{-1em}To (ii):
				The transition $\beta\downarrow 0$ can be conducted as in (i) by utilizing Lemma \ref{lemma:conv} (ii).
				\ep
		\end{enumerate}
	\end{proof}
	\textbf{Remarks to the proof of Theorem \ref{theorem:existenceCont}}
	\begin{enumerate}
		\item[(i)]
			The regularizing term $-\Delta\chi_t$ in \eqref{eqn:disc2} is needed in order
			to obtain an $H^1(H^1)$-bound for $\chi$ in the first estimate.
			This, in turn, was particularly necessary to estimate $\delta$ in \eqref{eqn:deltaEst}
			and to estimate the term $T_4$ in the second estimate.
		\item[(ii)]
			In the mathematical literature the elasticity equations \eqref{eqn:disc1} is sometimes
			tested with the function $-\DIV(\DD\e(u_t))$ to gain higher-order estimates for $u$
			(see \cite{WIAS1927,RR12,RR14}).
			However, due to the nonhomogeneous Neumann boundary condition \eqref{eqn:boundaryEq2}
			in our case, it is more convenient to test with
			$-\DIV\big(\mathsf{c}(\chi)\CC\e(u)+\mathsf{d}(\chi)\DD\e(u_t)\big)$ since, otherwise,
			integration by parts in space of the term $-\iint u_{tt}\cdot\DIV(\DD\e(u_t))$
			yields unpleasant terms even after using the boundary condition \eqref{eqn:boundaryEq2}
			(cf. estimates for $T_1$ in the second estimate).
	\end{enumerate}

\subsection{Continuous dependence on the data}
\label{sec:contDep}
	We are going to show continuous dependence on the data of the strong solutions of the PDE system given in Definition \ref{def:notionSolution} (i) and (ii).

	\begin{theorem}[Continuous dependence]
	\label{theorem:contDepCont}
		Let the Assumptions (A1)-(A4) be satisfied. Moreover, assume that $\mathsf{d}\equiv1$ in (A3) and one of the following condition:
		\begin{itemize}
			\item
				Let $(u_1,\chi_1)$ and $(u_2,\chi_2)$ be both strong solutions according to Definition \ref{def:notionSolution} (ii) (for $\beta>0$)
				with data $(u_1^0,v_1^0,\chi_1^0,b_1,\ell_1)$ and $(u_2^0,v_2^0,\chi_2^0,b_2,\ell_2)$.
			\item
				Let $(u_1,\chi_1)$ and $(u_2,\chi_2)$ be both strong solutions according to Definition \ref{def:notionSolution} (i) (for $\beta=0$)
				with data $(u_1^0,v_1^0,\chi_1^0,b_1,\ell_1)$ and $(u_2^0,v_2^0,\chi_2^0,b_2,\ell_2)$.
		\end{itemize}
		Then,
		\begin{align}
			&\|u_1-u_2\|_{W^{1,\infty}(0,T;L^2)\cap H^1(0,T;H^1)}
				+\|\chi_1-\chi_2\|_{H^1(0,T;H^1)}\notag\\
			&\leq C\Big(\|u_1^0-u_2^0\|_{H^1}+\|v_1^0-v_2^0\|_{L^2}
				+\|\chi_1^0-\chi_2^0\|_{H^1}+\|\ell_1-\ell_2\|_{L^2(0,T;L^2)}
				+\|b_1-b_2\|_{L^2(0,T;L^{2}(\Gamma;\Rn))}\Big),
		\label{eqn:lipEst}
		\end{align}		
		where the constant $C>0$ continuously depends on
		\begin{align*}
			C=C\Big(&\|u_1\|_{\C U},
				\|u_2\|_{\C U},
				\|\chi_{1}\|_{\C X},
				\|\chi_{2}\|_{\C X}
			\Big).
		\end{align*}
	\end{theorem}
	\begin{proof}
		For notational convenience, define
		\begin{align*}
			&\bu:=u_1-u_2,
			&&\bchi:=\chi_1-\chi_2,
			&&\buu:=u_1^0-u_2^0,\\
			&\bvv:=v_1^0-v_2^0,
			&&\bchii:=\chi_1^0-\chi_2^0,
			&&\bb:=b_1-b_2,\\
			&\bell:=\ell_1-\ell_2.
		\end{align*}
		Let $t\in[0,T]$ be arbitrary.
		Firstly, testing the damage equation \eqref{eqn:damageEq} for each solution
		with $\bchi_t$, subtracting the resulting equations and integrating over $\Omega\times(0,t)$, we obtain
		\begin{align}
			&\int_{0}^{t}\int_{\Omega}\Big(|\bchi_t|^2+\nabla\bchi\cdot\nabla\bchi_t+|\nabla\bchi_t|^2+\frac12 \mathsf{c}'(\chi_1)\big(\CC\e(u_1):\e(u_1)-\CC\e(u_2):\e(u_2)\big)\bchi_t
			\Big)\dxs\notag\\
			&+\int_{0}^{t}\int_{\Omega}\Big(\frac12\big(\mathsf{c}'(\chi_1)-\mathsf{c}'(\chi_2)\big)\CC\e(u_2):\e(u_2)\bchi_t+(f'(\chi_1)-f'(\chi_2))\bchi_t+(\xi_1-\xi_2)\bchi_t\Big)\dxs\notag\\
			&\qquad=0.
		\label{eqn:chiDiff}
		\end{align}
		By assumption, we know
		\begin{align*}
			&\xi_i\in\partial I_{(-\infty,0]}(\partial_t\chi_i),\;i=1,2&&\text{if }\beta=0,\\
			&\xi_i=I_\beta'(\partial_t\chi_i),\;i=1,2&&\text{if }\beta>0.
		\end{align*}
		It follows from the monotonicity of $\partial I_{(-\infty,0]}$ and $I_\beta'$ (see Definition \ref{def:regularization}), respectively,
		that
		$$
			(\xi_1-\xi_2)\bchi_t=(\xi_1-\xi_2)(\chi_1-\chi_2)_t\geq 0.
		$$
		Therefore, by \eqref{eqn:chiDiff},
		\begin{align}
			&\|\bchi_t\|_{L^2(0,t;H^1)}^2
				+\frac 12\|\nabla\bchi(t)\|_{L^2}^2
				-\frac 12\|\nabla\bchi^0\|_{L^2}^2\notag\\
			&\qquad\leq
				\underbrace{-\int_{0}^{t}\int_{\Omega}\frac 12 \mathsf{c}'(\chi_1)\CC\e(\bu):\e(u_1+u_2)\bchi_t\dxs}_{=:T_1}\notag\\
				&\qquad\quad\underbrace{-\int_{0}^{t}\int_{\Omega}\frac 12(\mathsf{c}'(\chi_1)-\mathsf{c}'(\chi_2))\CC\e(u_2):\e(u_2)\bchi_t\dxs}_{=:T_2}\notag\\
				&\qquad\quad\underbrace{-\int_{0}^{t}\int_{\Omega}(f'(\chi_1)-f'(\chi_2))\bchi_t\dxs}_{=:T_3}.
		\label{eqn:chiDiffTest}
		\end{align}
		By using H\"older's and Young's inequalities as well as standard Sobolev embeddings,
		the Lipschitz continuity of $\mathsf{c}'$ (see (A2)) and the Lipschitz continuity of $f'$ (see (A3), we find
		\begin{align*}
			T_1\leq{}&C\|\e(u_1+u_2)\|_{L^\infty(0,t;L^4)}\int_{0}^t\|\e(\bu)\|_{L^2}\|\bchi_t\|_{L^4}\ds\notag\\
				\leq{}&
					\delta\|\bchi_t\|_{L^2(0,t;H^1)}^2+C_\delta\|\e(\bu)\|_{L^2(0,t;L^2)}^2,\\
			T_2\leq{}&C\|\mathsf{c}'\|_{Lip}\|\e(u_2)\|_{L^\infty(0,t;L^4)}^2\int_{0}^t\|\bchi\|_{L^4}\|\bchi_t\|_{L^4}\ds\notag\\
				\leq{}&
					\delta\|\bchi_t\|_{L^2(0,t;H^1)}^2+C_\delta\|\bchi\|_{L^2(0,t;H^1)}^2,\\
			T_3\leq{}&\|f'\|_{Lip}\int_{0}^{t}\|\bchi\|_{L^2}\|\bchi_t\|_{L^2}\ds\\
				\leq{}&\delta\|\bchi_t\|_{L^2(0,t;L^2)}^2+C_\delta\|\bchi\|_{L^2(0,t;L^2)}^2.
		\end{align*}
		Applying the estimates for $T_1$, $T_2$ and $T_3$ to \eqref{eqn:chiDiffTest},
		we obtain
		\begin{align}
			\|\bchi_t\|_{L^2(0,t;H^1)}^2+\|\nabla\bchi(t)\|_{L^2}^2
			\leq{}& C_\delta\Big(\|\bchii\|_{H^1}^2+\|\bchi\|_{L^2(0,t;H^1)}^2+\|\bu\|_{L^2(0,t;H^{1})}^2\Big)
				+\delta\|\bchi_t\|_{L^2(0,t;H^1)}^2.
		\label{eqn:ineq1}
		\end{align}

		Secondly, we test each of the corresponding elasticity equations \eqref{eqn:elasticEq} for $u_1$ and $u_2$ with
		$\bu_t$ and obtain by subtraction and integration over $\Omega\times(0,t)$:
		\begin{align*}
			&\int_{0}^{t}\int_{\Omega}\Big(\bu_{tt}\bu_t+\mathsf{c}(\chi_1)\CC\e(\bu):\e(\bu_t)+(\mathsf{c}(\chi_1)-\mathsf{c}(\chi_2))\CC\e(u_2):\e(\bu_t)
				+\DD\e(\bu_t):\e(\bu_t)\Big)\dxs\\
			&\qquad=\int_{0}^{t}\int_{\Gamma}\bb\cdot\bu_t\dxs+\int_{0}^{t}\int_{\Omega}\bell\cdot\bu_t\dxs.
		\end{align*}
		This implies
		\begin{align}
			&\frac 12\|\bu_t(t)\|_{L^2}^2-\frac 12\|\bvv\|_{L^2}^2+\eta\|\e(\bu_t)\|_{L^2(0,t;L^2)}^2\notag\\
			&\qquad\leq\underbrace{-\int_{0}^{t}\int_{\Omega}\mathsf{c}(\chi_1)\CC\e(\bu):\e(\bu_t)\dxs}_{=:T_4}
				\underbrace{-\int_{0}^{t}\int_{\Omega}(\mathsf{c}(\chi_1)-\mathsf{c}(\chi_2))\CC\e(u_2):\e(\bu_t)\dxs}_{=:T_5}\notag\\
			&\qquad\quad
				+\underbrace{\int_{0}^{t}\int_{\Gamma}\bb\cdot\bu_t\dxs}_{=:T_6}
				+\underbrace{\int_{0}^{t}\int_{\Omega}\bell\cdot\bu_t\dxs}_{=:T_7}
		\label{eqn:uDiff}
		\end{align}
		Standard estimates yield
		\begin{align*}
			T_4&\leq \delta\|\e(\bu_t)\|_{L^2(0,t;L^2)}^2+C_\delta\|\e(\bu)\|_{L^2(0,t;L^2)}^2\\
				&\leq \delta\|\bu_t\|_{L^2(0,t;H^1)}^2+C_\delta\|\bu\|_{L^2(0,t;H^1)}^2,\\
			T_5&\leq C\|\bchi\|_{L^2(0,t;L^3)}\|\e(u_2)\|_{L^\infty(L^6)}\|\e(\bu_t)\|_{L^2(0,t;L^2)}\\
				&\leq \delta\|\bu_t\|_{L^2(0,t;H^1)}^2+C_\delta\|\bchi\|_{L^2(0,t;H^1)}^2,\\
			T_6&\leq\delta\|\bu_t\|_{L^2(0,t;L^2(\Gamma;\Rn))}^2+C_\delta\|\bb\|_{L^2(0,t;L^2(\Gamma;\Rn))}^2\\
				&\leq \delta\|\bu_t\|_{L^2(0,t;H^1)}^2+C_\delta\|\bb\|_{L^2(0,t;L^2(\Gamma;\Rn))}^2,\\
			T_7&\leq\delta\|\bu_t\|_{L^2(0,t;L^2)}^2+C_\delta\|\bell\|_{L^2(0,t;L^2)}^2.
		\end{align*}
		Applying the estimates $T_4$, $T_5$, $T_6$ and $T_7$ to \eqref{eqn:uDiff} shows
		\begin{align}
			&\|\bu_t(t)\|_{L^2}^2
			+\|\e(\bu_t)\|_{L^2(0,t;L^2)}^2\notag\\
				&\qquad\leq C_\delta\Big(
				\|\bvv\|_{L^2}^2
				+\|\bu\|_{L^2(0,t;H^1)}^2
				+\|\bchi\|_{L^2(0,t;H^1)}^2
				+\|\bell\|_{L^2(0,t;L^2)}^2\Big)\notag\\
				&\qquad\quad+C_\delta\|\bb\|_{L^2(0,t;L^2(\Gamma;\Rn))}^2+\delta\|\bu\|_{H^1(0,t;H^1)}^2.
			\label{eqn:ineq2}
		\end{align}
		Adding \eqref{eqn:ineq1} and \eqref{eqn:ineq2}, we obtain
		\begin{align}
			&\|\e(\bu_t)\|_{L^2(0,t;L^2)}^2
			+\|\bchi_t\|_{L^2(0,t;H^1)}^2
			+\|\bu_t(t)\|_{L^2}^2
			+\|\nabla\bchi(t)\|_{L^2}^2\notag\\
			&\qquad\leq C_\delta\Big(\|\bvv\|_{L^2}^2+\|\bchii\|_{H^1}^2
				+\|\bchi\|_{L^2(0,t;H^1)}^2
				+\|\bu\|_{L^2(0,t;H^1)}^2
				+\|\bell\|_{L^2(0,t;L^2)}^2
				+\|\bb\|_{L^2(0,t;L^2(\Gamma;\Rn))}^2
				\Big)\notag\\
				&\qquad\quad
				+\delta\Big(\|\bu\|_{H^1(0,t;H^1)}^2
				+\|\bchi_t\|_{L^2(0,t;H^1)}^2\Big).
		\label{eqn:chiUEst}
		\end{align}
		Now, adding $\|\bu\|_{L^2(0,t;H^1)}^2+\|\bu_t\|_{L^2(0,t;L^2)}^2$ on both sides and using
		\begin{align*}
			\|\bchi(t)\|_{L^2}^2=\|\bchii+\int_{0}^{t}\bchi_t(s)\ds\|_{L^2}^2
				\leq C\big(\|\bchii\|_{L^2}^2+\|\bchi_t\|_{L^2(0,t;L^2)}^2\big).
		\end{align*}
		and Korn's inequality
		$$
			\|w\|_{H^1}\leq C\big(\|w\|_{L^2}^2+\|\e(w)\|_{L^2}^2\big)
		$$
		holding for all $w\in H^1(\Omega;\Rn)$,
		the estimate \eqref{eqn:chiUEst} becomes
		\begin{align*}
			&\|\bu\|_{H^1(0,t;H^1)}^2
			+\|\bchi_t\|_{L^2(0,t;H^1)}^2
			+\|\bu_t(t)\|_{L^2}^2
			+\|\bchi(t)\|_{H^1}^2\notag\\
			&\qquad\leq C_\delta\Big(\|\bvv\|_{L^2}^2+\|\bchii\|_{H^1}^2+\|\bchi\|_{L^2(0,t;H^1)}^2
				+\|\bu\|_{H^1(0,t;L^2)}^2
				+\|\bu\|_{L^2(0,t;H^1)}^2\Big)\notag\\
				&\qquad\quad+C_\delta\Big(
				\|\bell\|_{L^2(0,t;L^2)}^2
				+\|\bb\|_{L^2(0,t;L^2(\Gamma;\Rn))}^2\Big)
				+\delta\Big(\|\bu\|_{H^1(0,t;H^1)}^2
				+\|\bchi_t\|_{L^2(0,t;H^1)}^2\Big).
		\end{align*}
		By choosing $\delta>0$ small and noticing
		$$
			\|\bu\|_{L^2(0,t;H^1)}^2\leq
				C\Big(\|\buu\|_{H^1}^2+\int_{0}^t\|\bu_t\|_{L^2(0,s;H^1)}^2\ds\Big),
		$$
		we get
		\begin{align*}
			&\|\bu\|_{H^1(0,t;H^1)}^2
			+\|\bchi_t\|_{L^2(0,t;H^1)}^2
			+\|\bu_t(t)\|_{L^2}^2
			+\|\bchi(t)\|_{H^1}^2\notag\\
			&\qquad\leq C\Big(\|\buu\|_{H^1}^2
				+\|\bvv\|_{L^2}^2
				+\|\bchii\|_{H^1}^2
				+\|\bell\|_{L^2(0,t;L^2)}^2
				+\|\bb\|_{L^2(0,t;L^2(\Gamma;\Rn))}^2\Big)\notag\\
				&\qquad\quad+C\int_{0}^{t}\Big(
				\|\bu_t\|_{L^2}^2
				+\|\bu_t\|_{L^2(0,s;H^1)}^2
				+\|\bchi\|_{H^1}^2
				\Big)\ds.
		\end{align*}
		The claim follows by Gronwall's lemma.
		\ep
	\end{proof}\\
	The continuous dependence result in Theorem \ref{theorem:contDepCont} as well as
	the a priori estimates in Lemma \ref{lemma:aPrioriDiscr} yield the
	following corollaries.
	For notational convenience we will make use of the spaces $\C U$ and $\C X$
	defined in Subsection \ref{sec:assumptions} (see \eqref{eqn:spaces}).

	\begin{corollary}[Uniqueness]
	\label{cor:uniqueness}
		Strong solutions in the sense of Definition \ref{def:notionSolution} (i) or (ii)
		with constant viscosity $\mathbb D$
		are unique to given initial-boundary data $(u^0,v^0,\chi^0,b,\ell)$.
	\end{corollary}
	
	\begin{corollary}[A~priori estimates]
	\label{cor:aPriori}
		A strong solution $(u,\chi)\in\C U\times\C X$ for the system in Definition \ref{def:notionSolution} (i) or (ii)
		with constant viscosity $\mathbb D$ and given data $(u^0,v^0,\chi^0,b,\ell)$
		satisfies the a priori estimates
		\begin{align*}
			&		\|u\|_{\C U}\leq C,
			&&	\|\chi\|_{\C X}\leq C,
			&&	\|\xi\|_{L^2(\Omega\times(0,T))}\leq C,
		\end{align*}
		where the constant $C>0$ continuously depends on
		$$
			C=C\big(\|u^0\|_{H^2},\|v^0\|_{H^1},\|\chi^0\|_{H^2}, \|b\|_{L^2(0,T;H^{1/2}(\Gamma;\Rn))\cap H^1(0,T;L^2(\Gamma;\Rn))},\|\ell\|_{L^2(0,T;L^2)}\big).
		$$
	\end{corollary}

	\section{Optimal control problem}
	\label{section:OCP}
	In this section we establish the announced optimal control problem for the damage-elasticity system \eqref{eqn:pdeSystem}-\eqref{eqn:IBC}.
	From now on, we assume for the viscosity tensor $\mathbb{D}=\DD$, i.e. $\mathsf{d}\equiv 1$, in order
	to apply the well-posedness result from the last section.
	
	Let $\C U$, $\C X$ and $\C B$ be given as in Section \ref{sec:assumptions}.
	Our aim is to approximate with $\chi$ prescribed damage profiles
	by controlling the Neumann boundary data $b\in\C B$ for
	the stress tensor $\sigma$ in \eqref{eqn:boundaryEq}.
	The cost functionals measures the deviation from the prescribed profiles
	at the final time or/and at all times during the evolution
	in an $L^\infty$-norm.
	We make the following assumptions:
	\begin{enumerate}
		\item[\textbf{(O1)}]
			We assume that $\lambda_Q$, $\lambda_\Omega$ and $\lambda_\Sigma$ are given non-negative constants which do not all vanish.
		\item[\textbf{(O2)}]
			The target damage profiles are given by
			\begin{align*}
				&\chi_Q\in L^\infty(Q),\qquad\chi_T\in L^\infty(\Omega).
			\end{align*}
		\item[\textbf{(O3)}]
			The admissible set of controls $\C B_{adm}\subseteq \C B$ is assumed to be non-empty, closed and bounded.
			Remember that $\C B:=L^2(0,T;H^{1/2}(\Gamma;\Rn))\cap H^1(0,T;L^{2}(\Gamma;\Rn))$ (see Subsection \ref{sec:assumptions}).
	\end{enumerate}
	\begin{remark}
		A typical choice for $\C B_{adm}$ would be
		$$
			\C B_{adm}=\big\{b\in\C B\;|\;b_{min}\leq b\leq b_{max}\text{ a.e. in $\Sigma$ and }\|b\|_{\C B}\leq M\big\},
		$$
		where $M\in(0,\infty)$ denotes the maximal $\C B$-cost and $b_{min}, b_{max}\in\C B$ the minimal and maximal cost functions satisfying $b_{min}\leq b_{max}$ a.e. in $\Sigma$.
	\end{remark}
	We define the following tracking type objective functional
	\begin{align}
	\label{eqn:costFunctional}
		\C J(\chi,b):={}&\frac{\lambda_Q}{2}\|\chi-\chi_Q\|_{L^\infty(Q)}
			+\frac{\lambda_\Omega}{2}\|\chi(T)-\chi_T\|_{L^\infty(\Omega)}
			+\frac{\lambda_\Sigma}{2}\|b\|_{L^2(\Sigma;\Rn)}^2,
	\end{align}
	where our overall optimization problem reads as
	\begin{align}
	\tag{$P_0$}
	\label{eqn:CP}
		\left.
		\begin{array}{l}
		 	\text{minimize }\C J(\chi,b)\text{ over }\C X\times\C B_{adm}\\
		 	\text{s.t. the PDE system in Definition \ref{def:notionSolution} (i) is satisfied for an $u\in\C U$.}
		\end{array}
		\right\}
	\end{align}
	\begin{remark}
		Let us emphasize that we may also choose $\|\cdot\|_{L^2}^2$-terms instead of the $\|\cdot\|_{L^\infty}$-terms in the cost functional \eqref{eqn:costFunctional}.
		The existence results presented in this section work for both cases.
	\end{remark}
	We recall that the system \eqref{eqn:pdeSystem}-\eqref{eqn:IBC} is an initial-boundary value problem,
	which admits by Theorem \ref{theorem:existenceCont} and Corollary \ref{cor:uniqueness} for every
	$(u^0,v^0,\chi^0,b,\ell)$ satisfying \eqref{eqn:initialData} and \eqref{eqn:externalForces}
	a unique solution $(u,\chi)\in \C U\times\C X$
	in the sense of Definition \ref{def:notionSolution} (i).
	Hence, the solution operator
	$$
		\Psi_0:\C I\to\C O,\qquad (u^0,v^0,\chi^0,b,\ell)\mapsto (u,\chi)
	$$
	with
	\begin{align*}
		&\C I:=\big\{(u^0,v^0,\chi^0,\ell,b)\;|\;\text{satisfying }\eqref{eqn:initialData}\text{ and }\eqref{eqn:externalForces}\big\},\\
		&\C O:=\C U\times\C X
	\end{align*}
	is well-defined. 
	Moreover, for fixed data $(u^0,v^0,\chi^0,\ell)$ the control-to-state operator
	$$
		S_0:\C B_{adm}\to\C O,\qquad b\mapsto(u,\chi)
	$$
	is also well-defined, and the optimal control problem \eqref{eqn:CP} is equivalent to minimizing the reduced cost functional
	$$
		j(b):=\C J (S_{0|2}(b),b)
	$$
	over $\C B_{adm}$,	where $S_{0|2}$ denotes the second component of $S_0$, i.e. $S_0=(S_{0|1},S_{0|2})$.
	
	For $\beta\in(0,1)$, let us denote by $S_\beta$ the operator mapping the control $b\in\C B_{adm}$
	into the unique solution $(u_\beta,\chi_\beta)\in\C O$ to the $\beta$-regularized problem
	in Definition \ref{def:notionSolution} (ii).
	\begin{remark}
		In view of the continuous dependence result in Theorem \ref{theorem:contDepCont} the operators
		$\Psi_0$ and $S_0$ are well-posed in a larger target space
		$\big(W^{1,\infty}(0,T;L^2(\Omega;\Rn))\cap H^1(0,T;H^1(\Omega;\Rn))\big)\times H^1(0,T;H^1(\Omega))$, which contains the space $\C O$.
		This fact is important for the sensitivity analysis of these operators. But in this section, we
		are interested only in existence of optimal controls, so this result is not needed.
		The sensitivity analysis which also establishes the optimality conditions of first-order will
		be treated in a forthcoming paper.     
	\end{remark}

	\subsection{Existence of optimal controls to \eqref{eqn:CP} via $\beta$-regularization}
	\label{section:existenceOCP}
	The following lemma is the basis for the main result in this section.
	\begin{lemma}
	\label{lemma:contProp}
		We have the following continuity properties:
		
		For a given sequence $\{b_\beta\}_{\beta\in(0,1)}\subseteq\C B$ and $\ol b\in\C B$ with
		\begin{align}
			b_\beta&\to\ol b\qquad\text{weakly in }\C B\text{ as }\beta\downarrow 0,
		\label{eqn:bBetaConv}
		\end{align}
		it holds
		\begin{subequations}
		\begin{align}
		\label{eqn:SConv1}
			&S_\beta(b_\beta)\to S_0(\ol b)\qquad\text{weakly-star in }\C O\text{ as }\beta\downarrow 0,\\
		\label{eqn:SConv2}
			&S_\eta(b_\beta)\to S_\eta(\ol b)\qquad\text{weakly-star in }\C O\text{ as }\beta\downarrow 0\text{ for every }\eta\in(0,1),\\
		\label{eqn:SConv3}
			&S_0(b_\beta)\to S_0(\ol b)\qquad\text{weakly-star in }\C O\text{ as }\beta\downarrow 0.
		\end{align}
		\end{subequations}
	\end{lemma}
	\begin{proof}
		Let $(u_\beta,\chi_\beta)=S_\beta(b_\beta)$. Then, $(u_\beta,\chi_\beta)$ is a solution to the
		$\beta$-regularized system in the sense of Definition \ref{def:notionSolution} (ii) with Neumann data $b_\beta$.
		Since $\{b_\beta\}\subseteq \C B$ is bounded by \eqref{eqn:bBetaConv}, we obtain the a priori estimates from Corollary \ref{cor:aPriori}.
		In particular,
		\begin{align*}
			&S_\beta(b_\beta)\to (u,\chi)\qquad\text{weakly-star in }\C O\text{ as }\beta\downarrow 0\text{ (for a subsequence)}.
		\end{align*}
		for some $(u,\chi)\in\C O$.
		
		We see from the proof of Lemma \ref{lemma:conv} that the convergence properties
		in Lemma \ref{lemma:conv} (ii) hold for a subsequence $\beta\downarrow 0$.
		By using these convergence properties as well as \eqref{eqn:bBetaConv},
		we can pass to the limit for a subsequence in the $\beta$-regularized PDE system \eqref{eqn:elasticRegEq}-\eqref{eqn:boundaryRegEq2}
		(cf. the proof of Theorem \ref{theorem:existenceCont}). We obtain that $(u,\chi)$ satisfies the limit system in Definition \ref{def:notionSolution} (i)
		to the Neumann data $\ol b$.
		In other words, $S_0(\ol b)=(u,\chi)$.
		By the uniqueness of solutions shown in Corollary \ref{cor:uniqueness}, we see that
		$S_\beta(b_\beta)$ convergences weakly-star to $(u,\chi)$ for the whole sequence $\beta\downarrow 0$.
		Hence, \eqref{eqn:SConv1} is shown and \eqref{eqn:SConv2} and \eqref{eqn:SConv3} follow with the same reasoning.
		\ep
	\end{proof}
	\begin{corollary}
	\label{cor:Jconv}
		We also have the property
		$$
			\lim_{\beta\downarrow 0}\C J(S_{\beta|2}(b),b)=\C J(S_{0|2}(b),b)
		$$
		for all $b\in\C B_{adm}$.
	\end{corollary}
	\begin{proof}
		Set $b_\beta:=b$ and apply Lemma \ref{lemma:contProp}. In particular, \eqref{eqn:SConv1} implies
		$$
			S_{\beta|2}(b)\to S_{0|2}(b)\quad\text{weakly-star in }\C X\text{ as }\beta\downarrow0.
		$$
		Define $\chi_\beta:=S_{\beta|2}(b)$ and $\chi:=S_{0|2}(b)$. A standard compactness
		result shows (see \cite{Simon})
		$$
			\chi_\beta\to \chi\quad\text{strongly in }C^0(\ol Q)\text{ as }\beta\downarrow0
		$$
		and the claim follows.
		\ep
	\end{proof}
	
	\begin{theorem}
	\label{theorem:optimalControl}
		Suppose that Assumptions (A1)-(A4) as well as (O1)-(O3) are satisfied.
		Then the optimal control problem \eqref{eqn:CP} admits a solution.
	\end{theorem}
	Before proving Theorem \ref{theorem:optimalControl}, we introduce a family of auxilliary optimal
	control problems (P$_\beta$), which are parametrized by $\beta\in(0,1)$.
	We define
	\begin{align}
	\tag{$P_\beta$}
	\label{eqn:CPbeta}
		\left.
		\begin{array}{l}
		 	\text{minimize }\C J(\chi,b)\text{ over }\C X\times\C B_{adm}\\
		 	\text{s.t. the $\beta$-regularized PDE system in Definition \ref{def:notionSolution} (ii) is satisfied.}
		\end{array}
		\right\}
	\end{align}
	The following result guarantees the existence of an optimal control to \eqref{eqn:CPbeta}.
	\begin{lemma}
	\label{lemma:CPbeta}
		Suppose that the Assumptions (A1)-(A4) as well as (O1)-(O3) are fulfilled.
		Let $\beta>0$ be given. Then the optimal control problem \eqref{eqn:CPbeta} admits a solution.
	\end{lemma}
	\begin{proof}
		Let $\{b^n\}_{n\in\N}\subseteq \C B_{adm}$ be a minimizing sequence for \eqref{eqn:CPbeta}, and let
		$(u_\beta^n,\chi_\beta^n)=S_\beta(b^n)$, $n\in\N$.
		By the boundedness and closedness of $\C B_{adm}$ (see (O3)), we find a function $\ol b\in\C B_{adm}$
		and a subsequence of $\{b^n\}$ (we omit the subscript) such that
		$$
			b^n\to\ol b\quad\text{weakly in }\C B\text{ as }n\uparrow\infty.
		$$
		Lemma \ref{lemma:contProp} yields
		$$
			S_\beta(b^n)\to S_\beta(\ol b)\quad\text{weakly-star in }\C O\text{ as }n\uparrow\infty.
		$$
		Let $(\ol u_\beta,\ol \chi_\beta):=S_\beta(\ol b)$. We particularly find
		$$
			\ol\chi_\beta^n=S_{\beta|2}(b^n)\to S_{\beta|2}(\ol b)=\ol\chi_\beta\quad\text{weakly-star in }\C X\text{ as }n\uparrow\infty.
		$$
		A standard compact result reveals
		$$
			\ol \chi_\beta^n\to \ol \chi_\beta\quad\text{strongly in }C^0(\ol Q)\text{ as }n\uparrow\infty.
		$$
		It follows from the sequentially weak lower semicontinuity of the cost functional $\C J$ that
		$\ol b$ is an optimal control for \eqref{eqn:CPbeta}, i.e.
		$$
			\C J(\ol\chi_\beta,\ol b_\beta)\leq \liminf_{n\uparrow\infty}\C J(\chi_\beta^n,b_\beta^n).
		$$
		\ep
	\end{proof}\\
	\textbf{Proof of Theorem \ref{theorem:optimalControl}.}
	By virtue of Lemma \ref{lemma:CPbeta}, for any $\beta\in(0,1)$, we may pick an optimality pair
	$$
		(\chi_\beta,b_\beta)\in\C X\times\C B_{adm}
	$$
	for the optimal control problem \eqref{eqn:CPbeta}.
	Obviously, we have $(u_\beta,\chi_\beta)=S_\beta(b_\beta)$, $\beta\in(0,1)$.
	By the assumption (O3) and Lemma \ref{lemma:contProp}, we find functions
	$(\ol u, \ol\chi)\in\C O$ and $\ol b\in\C B_{adm}$ with $S_0(\ol b)=(\ol u,\ol \chi)$ such that
	\begin{subequations}
	\label{eqn:CPconv}
	\begin{align}
		&&&&(u_\beta, \chi_\beta)&\to (\ol u,\ol \chi)&&\text{weakly-star in }\C O,&&&&\\
		&&&&b_\beta&\to \ol b&&\text{weakly in }\C B&&&&
	\end{align}
	\end{subequations}
	as $\beta\downarrow 0$ (for a subsequence).

	It remains to show that $(\ol\chi,\ol b)$ is in fact an optimality pair of \eqref{eqn:CP}.
	To this end, let $b\in\C B_{adm}$ be arbitrary.
	In view of the convergence properties \eqref{eqn:CPconv}
	and the sequentially weak lower semicontinuity of the cost functional, we have
	\begin{align*}
		\C J(\ol\chi,\ol b)
			\leq\liminf_{\beta\downarrow 0}\C J(\chi_\beta,b_\beta)
			=\liminf_{\beta\downarrow 0}\C J(S_{\beta|2}(b_\beta),b_\beta).
	\end{align*}
	By using the optimality property of \eqref{eqn:CPbeta}, we obtain
	\begin{align*}
		\liminf_{\beta\downarrow 0}\C J(S_{\beta|2}(b_\beta),b_\beta)
		\leq{}&\liminf_{\beta\downarrow 0}\C J(S_{\beta|2}(b),b).
	\end{align*}
	Finally, the convergence property in Corollary \ref{cor:Jconv} shows
	\begin{align*}
		\liminf_{\beta\downarrow 0}\C J(S_{\beta|2}(b),b)
		=\C J(S_{0|2}(b),b).
	\end{align*}
	In conclusion, we have proven $\C J(S_{0|2}(\ol b),\ol b)\leq \C J(S_{0|2}(b),b)$.
	\ep
	\begin{remark}
		Theorem \ref{theorem:optimalControl} can also be shown in the spirit of Lemma \ref{lemma:CPbeta}. However,
		the proof presented via convergence of $\beta$-approximations might be of interest in view of the implementation of optimality systems.
	\end{remark}

	\subsection{An adapted optimal control problem to \eqref{eqn:CP}}
	\label{section:adaptedOCP}
	Theorem \ref{theorem:optimalControl} does not yield any information on whether every solution
	to the optimal control problem \eqref{eqn:CP} can be approximated by a sequence of solutions to the problem
	\eqref{eqn:CPbeta}.
	As already announced in the introduction, we are not able to prove such
	a general ``global'' result because of the lack of uniqueness of optimizers.
	Hence we cannot guarantee that the weak limit of the sequence of optimizers
	to problem \eqref{eqn:CPbeta} converges always to the same optimizer of \eqref{eqn:CP}.
	Instead, we can only give an answer for every individual optimizer of \eqref{eqn:CP},
	which is the reason why this solution is called a ``local'' result (see also Remark
	\ref{remark:adaptedCP}).
	For this purpose, we employ a trick due to \cite[Section 5 - Proof of Theorem 1]{Bar81}.
	
	To this end, let $((\ol u,\ol\chi),\ol b )\in\C O\times\C B_{adm}$, where $(\ol u,\ol\chi)=S_0(\ol b)$,
	be an arbitrary but fixed solution to \eqref{eqn:CP}. We associate with this solution the
	\textit{adapted cost functional}
	$$
		\widetilde{\C J}(\chi,b):=\C J(\chi,b)+\frac12\|b-\ol b\|_{L^2(\Sigma;\Rn)}^2
	$$
	and the corresponding \textit{adapted optimal control problem}
	\begin{align}
	\tag{$\widetilde P_\beta$}
	\label{eqn:CPtildeBeta}
		\left.
		\begin{array}{l}
		 	\text{minimize }\widetilde{\C J}(\chi,b)\text{ over }\C X\times\C B_{adm}\\
		 	\text{s.t. the $\beta$-regularized PDE system in Definition \ref{def:notionSolution} (ii) is satisfied.}
		\end{array}
		\right\}
	\end{align}
	With a proof that resembles that of Lemma \ref{lemma:CPbeta} and needs no repetition here, we can show the following
	result:
	\begin{lemma}
	\label{lemma:CPtildeBeta}
		Suppose that the Assumptions (A1)-(A4) as well as (O1)-(O3) are fulfilled.
		Let $\beta\in(0,1)$ be given. Then, the optimal control problem \eqref{eqn:CPtildeBeta} admits a solution.
	\end{lemma}
	We are now in the position to give a partial answer to the question raised above.
	More precisely, we show the following theorem:
	\begin{theorem}
	\label{theorem:adaptedOCP}
		Let the Assumptions (A1)-(A4) and (O1)-(O3) be satisfied.
		Suppose that $(\ol\chi,\ol b)\in\C X\times \C B_{adm}$ is any fixed solution to the optimal control problem \eqref{eqn:CP}.
		Then, there exists a pair $(\ol\chi_\beta,\ol b_\beta)\in\C X\times\C B_{adm}$ solving the adapted problem \eqref{eqn:CPtildeBeta}
		such that
		$\widetilde{\C J}(\ol\chi_\beta,\ol b_\beta)\to \C J(\ol\chi,\ol b)$ as $\beta\downarrow 0$.
	\end{theorem}
	\begin{proof}
		For every $\beta\in(0,1)$ we pick an optimal pair $(\ol\chi_\beta,\ol b_\beta)\in\C X\times\C B_{adm}$
		for the adapted problem \eqref{eqn:CPtildeBeta}. By the boundedness and closedness of $\C B_{adm}$ (see (O3)), there exists a $b\in\C B_{adm}$
		satisfying
		\begin{align}
		\label{eqn:bConv}
			\ol b_\beta\to b\quad\text{weakly in }\C B\text{ as }\beta\downarrow0.
		\end{align}
		Owing to Lemma \ref{lemma:contProp} we find
		$$
			(\ol u_\beta,\ol\chi_\beta)=S_\beta(\ol b_\beta)\to S_0(b)=:(u,\chi)\quad\text{weakly-star in }\C O\text{ as }\beta\downarrow0.
		$$
		and, particularly,
		\begin{align}
		\label{eqn:chiBetaConv}
			\ol\chi_\beta\to \chi\quad\text{strongly in }C^0(\ol Q)\text{ as }\beta\downarrow0.
		\end{align}
		We now aim to prove that $b=\ol b$. Once this is shown, we can infer from the unique
		solvability of the state system (see Theorem \ref{theorem:contDepCont}) that also
		$(u,\chi)=(\ol u,\ol\chi)$.
		
		Indeed, we have, owing to \eqref{eqn:bConv}, \eqref{eqn:chiBetaConv}, the sequentially weak lower semicontinuity of $\widetilde{\C J}$,
		and the optimality property of $(\ol\chi,\ol b)$ for problem \eqref{eqn:CP},
		\begin{align}
			\liminf_{\beta\downarrow 0}\widetilde{\C J}(\ol\chi_\beta,\ol b_\beta)
			\geq{}& \C J(\chi,b)+\frac12\|b-\ol b\|_{L^2(\Sigma;\Rn)}^2\notag\\
			\geq{}& \C J(\ol\chi,\ol b)+\frac12\|b-\ol b\|_{L^2(\Sigma;\Rn)}^2.
		\label{eqn:liminfJTildeEst}
		\end{align}
		On the other hand, the optimality property of $(\ol\chi_\beta,\ol b_\beta)$ for problem
		\eqref{eqn:CPtildeBeta} yields that
		$$
			\widetilde{\C J}(\ol\chi_\beta,\ol b_\beta)=
			\widetilde{\C J}(S_{\beta|2}(\ol b_\beta),\ol b_\beta)
			\leq \widetilde{\C J}(S_{\beta|2}(\ol b),\ol b).
		$$
		Whence, taking the limes superior as $\beta\downarrow 0$ on both sides and invoking Corollary \ref{cor:Jconv}, we find
		\begin{align}
			\limsup_{\beta\downarrow 0}\widetilde{\C J}(\ol\chi_\beta,\ol b_\beta)
			\leq{}& \widetilde{\C J}(S_{0|2}(\ol b),\ol b)=\widetilde{\C J}(\ol\chi,\ol b)
			=\C J(\ol\chi,\ol b).
		\label{eqn:limsupJTildeEst}
		\end{align}
		We obtain by combining \eqref{eqn:liminfJTildeEst} and \eqref{eqn:limsupJTildeEst}
		$$
			\frac12\|b-\ol b\|_{L^2(\Sigma;\Rn)}^2=0.
		$$
		Thus $b=\ol b$ and, consequently, $(u,\chi)=(\ol u,\ol\chi)$ by Theorem \ref{theorem:contDepCont}.
		
		Finally, by using $b=\ol b$ in \eqref{eqn:liminfJTildeEst} and \eqref{eqn:limsupJTildeEst}, we end up with
		$\lim_{\beta\downarrow 0}\widetilde{\C J}(\ol\chi_\beta,\ol b_\beta)=\C J(\ol\chi,\ol b)$.
		\ep
	\end{proof}
	\begin{remark}
	\label{remark:adaptedCP}
		As we have seen in Theorem \ref{theorem:optimalControl},
		every weakly convergent subsequence of minimizers of \eqref{eqn:CPbeta}
		converges to a minimizer of \eqref{eqn:CP}.
		However, since the problem \eqref{eqn:CP} might not be uniquely solvable
		different subsequences of minimizers of \eqref{eqn:CPbeta} might converge
		to different minimizers of \eqref{eqn:CP}.
		By considering the adapted control problem \eqref{eqn:CPtildeBeta},
		we force the minimizers of \eqref{eqn:CPtildeBeta}
		to converge to the desired minimizer of \eqref{eqn:CP}
		as shown in Theorem \ref{theorem:adaptedOCP}.
		For further details to non-convex optimal control problems we refer to
		\cite{Bar81}.
	\end{remark}

	\addcontentsline{toc}{chapter}{Bibliography}{\footnotesize{\setlength{\baselineskip}{0.2 \baselineskip}
	\bibliography{references}
	}
	\bibliographystyle{plain}}
\end{document}